\pdfoutput=1
\RequirePackage{ifpdf}
\ifpdf 
\documentclass[pdftex]{sigma}
\else
\documentclass{sigma}
\fi

\numberwithin{equation}{section}

\newtheorem{Theorem}{Theorem}[section]
\newtheorem{Corollary}[Theorem]{Corollary}
\newtheorem{Lemma}[Theorem]{Lemma}
\newtheorem{Proposition}[Theorem]{Proposition}
 { \theoremstyle{definition}
\newtheorem{Example}[Theorem]{Example}
\newtheorem{Remark}[Theorem]{Remark} }

\usepackage{bm}

\DeclareMathAlphabet\EuRoman{U}{eur}{m}{n}
\SetMathAlphabet\EuRoman{bold}{U}{eur}{b}{n}
\newcommand{\eurom}{\EuRoman}

\renewcommand{\d}{\mathrm{d}}

\newcommand{\cI}{{\mathcal I}}
\newcommand{\cJ}{{\mathcal J}}
\newcommand{\cV}{{\mathcal V}}
\newcommand{\bbR}{{\mathbb R}}

\newcommand{\bbP}{{\mathbb P}}

\newcommand{\GL}{\operatorname{GL}}

\newcommand{\SO}{\operatorname{SO}}

\newcommand{\Ric}{\operatorname{Ric}}

\newcommand{\euso}{\operatorname{\frak{so}}}
\newcommand{\eugl}{\operatorname{\frak{gl}}}

\newcommand{\mba}{\mathbf{a}}
\newcommand{\mbb}{\mathbf{b}}
\newcommand{\mbe}{\mathbf{e}}
\newcommand{\mbg}{\mathbf{g}}
\newcommand{\mbf}{\mathbf{f}}
\newcommand{\mbh}{\mathbf{h}}
\newcommand{\mbk}{\mathbf{k}}

\newcommand{\mbu}{\mathbf{u}}
\newcommand{\mbv}{\mathbf{v}}
\newcommand{\mbw}{\mathbf{w}}
\newcommand{\mbx}{\mathbf{x}}
\newcommand{\mby}{\mathbf{y}}
\newcommand{\mbz}{\mathbf{z}}
\newcommand{\euroa}{\bm{\eurom{a}}}
\newcommand{\eurof}{\bm{\eurom{f}}}
\newcommand{\eurox}{\bm{\eurom{x}}}

\DeclareMathOperator{\tr}{tr}

\DeclareMathOperator{\Hom}{Hom}

\newcommand{\phm}{\phantom{-}}

\newcommand{\lhk}{\mathbin{\hbox{\vrule height1.4pt width4pt depth-1pt
 \vrule height4pt width0.4pt depth-1pt}}}

\begin{document}
\allowdisplaybreaks

\newcommand{\arXivNumber}{1908.01041}

\renewcommand{\PaperNumber}{004}

\FirstPageHeading

\ShortArticleName{Flat Metrics with a Prescribed Derived Coframing}

\ArticleName{Flat Metrics with a Prescribed Derived Coframing}

\Author{Robert L.~BRYANT~$^\dag$ and Jeanne N.~CLELLAND~$^\ddag$}

\AuthorNameForHeading{R.L.~Bryant and J.N.~Clelland}

\Address{$^\dag$~Duke University, Mathematics Department, P.O. Box 90320, Durham, NC 27708-0320, USA}
\EmailD{\href{mailto:bryant@math.duke.edu}{bryant@math.duke.edu}}

\Address{$^\ddag$~Department of Mathematics, 395 UCB, University of Colorado, Boulder, CO 80309-0395, USA}
\EmailD{\href{mailto:Jeanne.Clelland@colorado.edu}{Jeanne.Clelland@colorado.edu}}

\ArticleDates{Received August 28, 2019, in final form January 09, 2020; Published online January 20, 2020}

\Abstract{The following problem is addressed: A $3$-manifold~$M$ is endowed with a~triple $\Omega = \big(\Omega^1,\Omega^2,\Omega^3\big)$ of closed $2$-forms. One wants to construct a coframing~$\omega = \big(\omega^1,\omega^2,\omega^3\big)$ of~$M$ such that, first, ${\rm d}\omega^i = \Omega^i$ for $i=1,2,3$, and, second, the Riemannian metric $g=\big(\omega^1\big)^2+\big(\omega^2\big)^2+\big(\omega^3\big)^2$ be flat. We show that, in the `nonsingular case', i.e., when the three $2$-forms~$\Omega^i_p$ span at least a $2$-dimensional subspace of~$\Lambda^2(T^*_pM)$ and are real-analytic in some $p$-centered coordinates, this problem is always solvable on a neighborhood of~$p\in M$, with the general solution~$\omega$ depending on three arbitrary functions of two variables. Moreover, the characteristic variety of the generic solution~$\omega$ can be taken to be a nonsingular cubic. Some singular situations are considered as well. In particular, we show that the problem is solvable locally when $\Omega^1$, $\Omega^2$, $\Omega^3$ are scalar multiples of a single 2-form that do not vanish simultaneously and satisfy a nondegeneracy condition. We also show by example that solutions may fail to exist when these conditions are not satisfied.}


\Classification{53A55; 53B15}

\section{Introduction}\label{sec: intro}

\subsection{The problem} Given a $3$-manifold~$M$
and a triple~$\Omega = \big(\Omega^1,\Omega^2,\Omega^3\big)$
of closed $2$-forms on~$M$,
it is desired to find a coframing~$\omega = \big(\omega^1,\omega^2,\omega^3\big)$
(i.e., a triple of linearly independent $1$-forms)
satisfying the first-order differential equations
\begin{gather}\label{eq: first-order DEs}
\d\omega^i = \Omega^i
\end{gather}
and the second-order equations that ensure that the metric
\begin{gather}\label{eq: definition of g}
g = \big(\omega^1\big)^2+\big(\omega^2\big)^2+\big(\omega^3\big)^2
\end{gather}
be flat.

This question was originally posed in the context of a problem regarding `residual stress' in elastic bodies due to defects, where the existence of solutions to equations \eqref{eq: first-order DEs} and \eqref{eq: definition of g} is related to the existence of residually stressed bodies that also satisfy a global energy minimization condition. (See~\cite{Acharya} for more details.) However, we feel that the problem is of independent geometric interest.

\subsection{Initial discussion}
As posed, this problem becomes an overdetermined system
of equations for the coframing~$\omega$,
which, in local coordinates~$\big(u^1,u^2,u^3\big)$, can be specified
by choosing the $9$ coefficient functions~$a^i_j(u)$
in the expansion~$\omega^i = a^i_j(u) \d u^j$.
Indeed,~\eqref{eq: first-order DEs}
is a system of $9$ first-order equations
while the flatness of the metric $g$
as defined in~\eqref{eq: definition of g}
is the system of $6$ second-order equations $\Ric(g) = 0$.
Together, these constitute a system of $15$
partial differential equations on the coefficients $a^i_j$
that are independent in the sense that no one of them is a combination of derivatives of the others.

However, the problem can be recast into a different form that makes it more tractable. For simplicity, we will assume that $M$ is connected and simply-connected. The condition that the $\mathbb{R}^3$-valued $1$-form~$\omega$ define a flat metric $g = {}^t\omega \circ \omega$ is then well-known to be equivalent to the condition that~$\omega$
be representable as
\begin{gather*}
\omega = \mba^{-1} \d \mbx,
\end{gather*}
where $\mbx\colon M\to\mathbb{R}^3$ is an immersion and $\mba\colon M\to\SO(3)$ is a smooth mapping.\footnote{In this note, we regard $\mathbb{R}^3$ as \emph{columns}
of real numbers of height $3$, though we will, from time to time, without comment, write them as row vectors in the text.}
This representation is unique up to a replacement of the form
\begin{gather*}
(\mbx,\mba)\mapsto (\mbx',\mba') = (R\mbx + T, R\mba),
\end{gather*}
where $T\in\mathbb{R}^3$ is a constant and $R\in\SO(3)$ is a constant.

Since $\SO(3)$ has dimension~$3$, specifying a pair $(\mbx,\mba)\colon M\to\mathbb{R}^3\times\SO(3)$ is, locally, a choice of $6$ arbitrary (smooth) functions on~$M$.
The remaining conditions on~$\omega$ needed to solve our problem,
\begin{gather}\label{eq: 2_form_eqs}
\d\big(\mba^{-1} \d \mbx\big) = -\mba^{-1} \d \mba \wedge \mba^{-1} \d \mbx = \Omega,
\end{gather}
still constitute $9$ independent first-order equations
for the `unknowns'~$(\mbx,\mba)$ (which are essentially~$6$ in number),
but these equations are not fully independent:
 $\d\Omega=0$ by hypothesis,
and the exterior derivatives of the three $2$-forms
on the left hand side of~\eqref{eq: 2_form_eqs}
also vanish identically for any pair~$(\mbx,\mba)$,
which provides $3$ `compatibility conditions' for the $9$ equations,
thereby, at least formally,
restoring the `balance' of $6$ equations for $6$ unknowns.
Thus, this rough count gives some indication that the problem might
be locally solvable.

However, caution is warranted.
Let $(\bar \mbx,\bar \mba)\colon M\to\mathbb{R}^3\times\SO(3)$
be a~smooth mapping and let~$\bar\Omega = \d\big(\bar \mba^{-1} \d \bar \mbx\big)$.
Linearizing the equations~\eqref{eq: 2_form_eqs}
at the `solution' $(\mbx,\mba) = (\bar \mbx,\bar \mba)$
yields a~system of differential equations of the form
\begin{gather}\label{eq: linearizedeqs}
\d\big( \bar \mba^{-1} (\d \mby - \mbb \d\bar \mbx )\big) = \Psi,
\end{gather}
where \looseness=-1 $(\mby,\mbb)\colon M\to \mathbb{R}^3\oplus \euso(3)$ are unknowns
and $\Psi$ is a closed $2$-form with values in $\mathbb{R}^3$.
If one were expecting~\eqref{eq: 2_form_eqs} to always be solvable,
one might na{\"\i}vely expect \eqref{eq: linearizedeqs} to always
be solvable as well, but this is not so:
When one linearizes at $(\bar \mbx,\bar \mba) = (\bar \mbx, I_3)$,
the linearized system reduces to
\begin{gather}\label{eq: linearizedeqsa=I}
 - \d \mbb\wedge\d\bar \mbx = \Psi,
\end{gather}
where $\mbb\colon M\to\euso(3)\simeq\mathbb{R}^3$
is essentially a set of $3$ unknowns
and $\Psi$ is a given closed $2$-form with values in $\mathbb{R}^3$.
However, as is easily seen, the solvability of~\eqref{eq: linearizedeqsa=I}
for $\mbb$ imposes a system of~$9$ independent first-order
linear equations on~$\Psi$, while the closure of $\Psi$
is only a subsystem of $3$ independent first-order linear equations on~$\Psi$.

Thus, some care needs to be taken in analyzing the system.
Indeed, as Example~\ref{ex: nonsolvability} in~Section~\ref{sec: Zrank1} shows,
there exists an~$\Omega$ defined on a neighborhood of the origin in~$\mathbb{R}^3$
for which there is no solution~$\omega = \mba^{-1} \d \mbx$
to the system~\eqref{eq: 2_form_eqs} on an open neighborhood of the origin.

\subsection{An exterior differential system}
The above observation suggests formulating the problem
as an exterior differential system~$\cI$
on~$X = M\times\mathbb{R}^3\times\SO(3)$
that is generated by the three $2$-form components
of the closed $2$-form
\begin{gather}\label{eq: Theta_defined}
\Theta = -\euroa^{-1} \d \euroa \wedge \euroa^{-1} \d \eurox - \Omega,
\end{gather}
where now, one regards $\eurox\colon X\to\mathbb{R}^3$ and $\euroa\colon X\to\SO(3)$
as projections on the second and third factors.\footnote{We use a different font in equation \eqref{eq: Theta_defined} to emphasize that $\euroa$, $\eurox$, etc., denote matrix- and vector-valued coordinate functions on $X$, while $\mba$, $\mbx$, etc., denote matrix- and vector-valued functions on $M$. We use $\Omega$ to denote both the 2-form on $\mathbb{R}^3$ and its pullback to $X$ via the projection map $\eurox\colon X\to\mathbb{R}^3$.}

We will show that, when~$\Omega$ is suitably nondegenerate,
this exterior differential system is involutive, i.e.,
it possesses Cartan-regular integral flags at every point.
In particular, if~$\Omega$ is also real-analytic,
the Cartan--K\"ahler theorem will imply that the original problem
is locally solvable.

\subsection{Background}
For the basic concepts and results
from the theory of exterior differential systems
that will be needed in this article,
the reader may consult Chapter III of~\cite{BCGGG}. The book \cite{CFB} may also be of interest.

\section{Analysis of the exterior differential system}\label{EDS-sec}

\subsection{Notation}
Define an isomorphism~$[{\cdot}]\colon \mathbb{R}^3\to\euso(3)$
(the space of $3$-by-$3$ skew-symmetric matrices) by the formula
\begin{gather*}
\left[ \mbx \right]
= \left[\begin{pmatrix}x^1\\x^2\\x^3\end{pmatrix}\right]
= \begin{pmatrix}
 0 & \phm x^3 & -x^2\\
 -x^3 & 0 & \phm x^1\\
 \phm x^2 & -x^1 & 0
 \end{pmatrix}.
\end{gather*}
The identity $[\mba \mbx] = \mba[\mbx]\mba^{-1}$,
which holds for all $\mba\in\SO(3)$ and $\mbx\in\mathbb{R}^3$,
will be useful, as will the following identities
for $\mbx, \mby\in\mathbb{R}^3$;
$A$ a $3$-by-$3$ matrix with real entries;
$\alpha$ and $\beta$ $1$-forms with values in $\mathbb{R}^3$;
and $\gamma$ a $1$-form with values in $3$-by-$3$ matrices:
\begin{gather}
 [\mbx] \mby = - [\mby] \mbx,\nonumber\\
[A\mbx] = (\tr A) [\mbx] - {}^t\!A [\mbx] - [\mbx] A,\nonumber\\
[\mbx][\mby] = \mby {}^t\mbx - {}^t\mbx \mby I_3,\nonumber\\
[\alpha ]\wedge\beta = [\beta]\wedge\alpha,\nonumber\\
[\gamma\wedge\alpha] = (\tr\gamma)\wedge[\alpha]-{}^t\gamma\wedge[\alpha]+[\alpha]\wedge\gamma,\nonumber\\
[\alpha]\wedge [\beta] = {}^t\beta\wedge\alpha I_3 - \beta\wedge {}^t\alpha,\nonumber\\
{}^t\alpha\wedge[\alpha]\wedge\alpha = - 6 \alpha^1\wedge\alpha^2\wedge\alpha^3,\nonumber\\
[A \alpha]\wedge\alpha = \tfrac12\bigl((\tr A) I_3 - {}^t\!A\bigr) [\alpha]\wedge\alpha.\label{eq: crossproductidentities}
\end{gather}
There is one more identity along these lines that will be useful. It is valid for all $\mathbb{R}^3$-valued $1$-forms~$\alpha$ and functions~$A$ with values in $\GL(3,\mathbb{R})$:
\begin{gather*}
[A\alpha]\wedge A\alpha = \det(A) \big({}^t\!A\big)^{-1} [\alpha]\wedge\alpha.
\end{gather*}

On~$\bbR^3\times\SO(3)$ with first and second factor projections~$\eurox\colon \bbR^3\times\SO(3)\to\bbR^3$ and $\euroa\colon \bbR^3\times\SO(3)\to\SO(3)$,
define the $\mathbb{R}^3$-valued $1$-forms $\xi$ and $\alpha$ by
\begin{gather} \label{basis-forms-on-R3xSO3}
\xi = \euroa^{-1} \d \eurox
\qquad\text{and}\qquad
[\alpha] = \euroa^{-1} \d \euroa = \begin{pmatrix}
0 & \phm\alpha^3 & -\alpha^2\\
-\alpha^3 & 0 & \phm\alpha^1\\
\phm\alpha^2 & -\alpha^1 & 0
\end{pmatrix}.
\end{gather}
These $1$-forms satisfy the so-called `structure equations', i.e., the identities
\begin{gather} \label{structure-eqs-on-R3xSO3}
\d\xi = -[\alpha]\wedge\xi
\qquad\text{and}\qquad
\d\alpha = -\tfrac12 [\alpha]\wedge\alpha.
\end{gather}

\subsection{Formulation as an exterior differential systems problem}
Now suppose that, on~$M^3$, there is specified an $\mathbb{R}^3$-valued, closed $2$-form~$\Omega=\big(\Omega^i\big)$. Choose an $\mathbb{R}^3$-valued coframing~$\eta = (\eta^i)\colon TM\to\bbR^3$.
Then one can write
\begin{gather*}
\Omega = \tfrac12 Z [\eta]\wedge\eta ,
\end{gather*}
where $Z$ is a function on~$M$ with values in $3$-by-$3$ matrices.

Let $\cI$ be the exterior differential system on $X^9 = M\times\mathbb{R}^3\times\SO(3)$ that is generated by the three components of the closed $2$-form
\begin{gather*}
\Theta = \d\xi - \Omega = -[\alpha]\wedge\xi - \tfrac12 Z [\eta]\wedge\eta.
\end{gather*}

\begin{Proposition}\label{prop: formulation_as_EDS} If $N^3\subset X$ is an integral manifold of~$\cI$ to which~$\eta$ and~$\xi$ pull back to be coframings, then each point of~$N^3$ has an open neighborhood that can be written as a graph
\begin{gather}\label{eq: xa-graphinX}
\bigl\{\bigl(p,\mbx(p),\mba(p)\bigr)\, \vrule \, p\in U \bigr\}\subset X
\end{gather}
for some open set~$U\subset M$ and smooth maps~$\mbx\colon U\to\mathbb{R}^3$
and $\mba\colon U\to\SO(3)$. Moreover, on~$U$,
the coframing~$\omega = \mba^{-1} \d \mbx$ satisfies~$\d\omega=\Omega$
and the metric~$g = {}^t\omega\circ \omega = {}^t\d \mbx\circ \d \mbx$ is flat.

Conversely, if $U\subset M$ is a simply-connected open subset on which there exists a coframing~$\omega\colon TU\to\bbR^3$ satisfying $(i)$ $\d\omega=\Omega$, and $(ii)$ the metric~$g={}^t\omega\circ \omega$ be flat, then there exist mappings~$\mbx\colon U\to\bbR^3$ and $\mba\colon U\to\SO(3)$ such that $\omega = \mba^{-1} \d \mbx$. Moreover, the immersion $\iota\colon U\to X$ defined by $\iota(p) = \bigl(p,\mbx(p),\mba(p)\bigr)$ is an integral manifold of~$\cI$ that pulls~$\eta$ and~$\xi$ back to be coframings of~$U$.
\end{Proposition}

\begin{proof} The statements in the first paragraph of the proposition are proved by simply unwinding the definitions and can be left to the reader.

For the converse statements (i.e., the second paragraph), suppose that a coframing $\omega\colon TU\to\bbR^3$ be given satisfying the two conditions. By the fundamental lemma of Riemannian geometry, there exists a unique $\bbR^3$-valued $1$-form~$\phi\colon TU\to\bbR^3$ such that
\begin{gather*}
\d\omega = -[\phi]\wedge\omega.
\end{gather*}
The condition that the metric~$g = {}^t\omega\circ \omega$ be flat is then the condition that~$\d\phi = -\tfrac12 [\phi]\wedge\phi$. These equations for the exterior derivatives of $\omega$ and $\phi$, together with the simple-connectivity of~$U$, imply that there exist maps~$\mbx\colon U\to\bbR^3$ and $\mba\colon U\to\SO(3)$ such that
\begin{gather}\label{flat-omega-and-phi}
\omega = \mba^{-1} \d \mbx
\qquad\text{and}\qquad
[\phi] = \mba^{-1} \d \mba.
\end{gather}
Consequently, $g = {}^t\omega\circ \omega$ is equal to ${}^t\d \mbx\circ \d \mbx$, which is flat, by definition. Finally, since~$\d\omega = \Omega$, it follows that the graph manifold~$N^3\subset X$ defined by~\eqref{eq: xa-graphinX} is an integral manifold of~$\cI$. Moreover, since, by construction,
\begin{gather*}
(\mathrm{id}_U, \mbx ,\mba)^*(\xi) = \omega,
\end{gather*}
it follows that $\xi$ and $\eta$ pull back to~$N^3$ to be coframings on~$N^3$.
\end{proof}

\begin{Remark}Observe that the $1$-forms $\omega$ and $\phi$ in equation \eqref{flat-omega-and-phi} are the pullbacks to $U$ of the $1$-forms $\xi$ and $\alpha$, respectively, on $\bbR^3 \times \SO(3)$ defined by equation \eqref{basis-forms-on-R3xSO3}. We will continue to use this notation to distinguish between forms on $\bbR^3 \times \SO(3)$ and their pullbacks via 3-dimensional immersions throughout the paper.
\end{Remark}

\subsection{Integral elements}
By Proposition~\ref{prop: formulation_as_EDS},
proving existence of local solutions of our problem
is equivalent to proving the existence of integral manifolds of~$\cI$
to which $\xi$ and $\eta$ pull back to be coframings. (This latter
condition is usually referred to as an `independence condition'.)

The first step in this approach
is to understand the nature of the integral elements of~$\cI$,
i.e., the candidates for tangent spaces to the integral manifolds of~$\cI$.

A (necessarily $3$-dimensional) integral element~$E\in\mathrm{Gr}(3,TX)$
of~$\cI$ will be said to be \emph{admissible}
if both $\xi\colon E\to\bbR^3$ and $\eta\colon E\to\bbR^3$ are isomorphisms.

\begin{Proposition}\label{prop: K-ordinaryintegralelements}
All of the admissible integral elements of~$\cI$ are K\"ahler-ordinary.\footnote{For definitions of K\"ahler-ordinary, Cartan-ordinary, etc.,
see~\cite[Chapter~III, Definition~1.7]{BCGGG}.}
The set $\cV_3\bigl(\cI,(\xi,\eta)\bigr)$
consisting of admissible integral elements of~$\cI$
is a submanifold of~$\mathrm{Gr}(3,TX)$,
and the basepoint projection~$\cV_3\bigl(\cI,(\xi,\eta)\bigr)\to X$
is a surjective submersion
with all fibers diffeomorphic to~$\GL(3,\bbR)$.
\end{Proposition}

\begin{proof}
Let $(p,\eurox,\euroa)\in X=M\times\bbR^3\times\SO(3)$, and
let $E\subset T_{(p,\eurox,\euroa)} X$ be a $3$-dimensional integral element of $\cI$
to which both $\xi$ and $\eta$ pull back to give an isomorphism of~$E$
with $\mathbb{R}^3$. Then there will exist a $P\in\GL(3,\bbR)$
and a $3$-by-$3$ matrix~$Q$ with real entries
such that $E\subset T_{(p,\eurox,\euroa)} X$ is defined as the kernel of the surjective linear mapping
\begin{gather}\label{eq: Eperp_eqns}
(\xi{-}P \eta, \alpha{-}QP \eta)\colon \ T_{(p,\eurox,\euroa)}\to\bbR^3\oplus\bbR^3.
\end{gather}
To simplify the notation, set $\bar\eta = E^*\eta$.
Then, $E^*\xi = P\bar\eta$ and $E^*\alpha = QP\bar\eta$.
The $2$-form~$\Theta$, which vanishes
when pulled back to~$E$, becomes
\begin{gather*}
0 = E^*\Theta= -[QP \bar\eta]\wedge P \bar\eta-\tfrac12 Z(p) [\bar\eta]\wedge\bar\eta\\
\hphantom{0 = E^*\Theta}{} = -\tfrac12\bigl(\bigl((\tr Q)I_3-{}^t\!Q\bigr)\det(P)\big({}^t\!P\big)^{-1}
 + Z(p)\bigr) [\bar\eta]\wedge\bar\eta.
\end{gather*}
Since~$\bar\eta:E\to\bbR^3$ is an isomorphism, it follows that
\begin{gather*}
\bigl((\tr Q)I_3 - {}^t\!Q\bigr) + Z(p) \,{}^t\!P/\det(P) = 0,
\end{gather*}
so that, solving for~$Q$, one has
\begin{gather}\label{eq: QintermsofPandZ}
Q = \det(P)^{-1}\big(P \, {}^t\!Z(p) - \tfrac12 \tr\bigl(P\, {}^t\!Z(p)\bigr) I_3\big).
\end{gather}

Conversely, if $(p,\eurox,\euroa)\in X=M\times\bbR^3\times\SO(3)$
and~$P\in\GL(3,\bbR)$ are arbitrary
and one defines $Q$ via~\eqref{eq: QintermsofPandZ},
then the kernel~$E\subset T_{(p,\eurox,\euroa)}X$
of the mapping~\eqref{eq: Eperp_eqns}
is an admissible integral element of~$\cI$.

The claims of the Proposition follow directly from these observations.
\end{proof}

\subsection{Polar spaces and Cartan-regularity}
In order to be able to apply the Cartan--K\"ahler theorem to prove existence of solutions in the real-analytic category, one needs a stronger result than Proposition~\ref{prop: K-ordinaryintegralelements}; one needs to show that there are \emph{Cartan}-ordinary admissible integral elements, in other words, to establish the existence of ordinary \emph{flags} terminating in elements of~$\cV_3\bigl(\cI,(\xi,\eta)\bigr)$. This requires some further investigations of the structure of the ideal~$\cI$ near a given integral element in~$\cV_3\bigl(\cI,(\xi,\eta)\bigr)$.

Let $E\in\cV_3\bigl(\cI,(\xi,\eta)\bigr)$ be fixed, with $E\subset T_{(p,\eurox,\euroa)}X$, and let $E$ be defined in this tangent space by the $6$ linear equations
\begin{gather}\label{eq: E in PQterms}
\xi - P \eta = \alpha - QP \eta = 0,
\end{gather}
where $Q$ is given in terms of $P\in\mathrm{GL}(3,\mathbb{R})$ and $Z(p)$ by~\eqref{eq: QintermsofPandZ}. For simplicity, set $\xi_E = (\xi - P \eta)_{\vrule(p,\eurox,\euroa)}$
and $\alpha_E = (\alpha - QP \eta)_{\vrule(p,\eurox,\euroa)}$, and let $\omega_E = (P \eta)_{\vrule(p,\eurox,\euroa)}$. The $9$ components of $\xi_E$, $\alpha_E$, and $\omega_E$ yield a basis of $T^*_{(p,\eurox,\euroa)}X$, with $E^\perp\subset T^*_{(p,\eurox,\euroa)}X$ being spanned by the components of~$\xi_E$ and $\alpha_E$ while $\omega_E\colon E\to\bbR^3$ is an isomorphism.

After calculation using~\eqref{eq: QintermsofPandZ} and the identities~\eqref{eq: crossproductidentities}, one then finds that $\Theta_{\vrule(p,\eurox,\euroa)}$ has the following expression in terms of $\xi_E$, $\alpha_E$, and $\omega_E$:
\begin{gather*}
\Theta_{\vrule(p,\eurox,\euroa)} = - [\alpha_E ]\wedge \omega_E
- [Q \omega_E ]\wedge \xi_E - [\alpha_E ]\wedge \xi_E\\
\hphantom{\Theta_{\vrule(p,\eurox,\euroa)}}{} =-\bigl( [\alpha_E ]+ [\xi_E ]Q\big)\wedge\omega_E
 - \left[\alpha_E\right]\wedge \xi_E .
\end{gather*}
The second term in this final expression, $- [\alpha_E ]\wedge \xi_E$, lies in $\Lambda^2\big(E^\perp\big)$ and hence plays no role in the calculation of the polar equations of~$E$. Hence, the polar spaces for an integral flag of~$E$ can be calculated using only $-\bigl( [\alpha_E ]+ [\xi_E ]Q\big)\wedge\omega_E$.

If $(\mbe_1,\mbe_2,\mbe_3)$ is a basis of~$E$, let $E_i\subset E$ be the subspace spanned by $\{ \mbe_j\, \vrule \, j\le i \}$ and set $w_i = \omega_E(\mbe_i)\in\bbR^3$.
Then the polar space of~$E_i$ is given by
\begin{gather*}
H(E_i) = \big\{\mbv\in T_{(p,\eurox,\euroa)}X\, \vrule\, \big( [\alpha_E(\mbv) ]+ [\xi_E(\mbv) ]Q\big)w_j =0, \ j\le i\big\}.
\end{gather*}
Consequently, the codimension~$c_i$ of this polar space
satisfies $c_i \le 3i$ for $0\le i\le 3$. Since the codimension
of~$\cV_3\bigl(\cI,(\xi,\eta)\bigr)$ in~$\mathrm{Gr}(3,TX)$ is~$9$,
which is always greater than or equal to $c_0 + c_1 + c_2$,
it follows, by Cartan's test, that the flag
$(E_0\subset E_1\subset E_2\subset E_3)$
will be Cartan-ordinary if and only if $c_0+c_1+c_2=9$,
i.e., $c_i = 3i$ for $i = 0,1,2$.
Moreover, this holds if and only if $c_2 = 6$.

Whether or not there is a $2$-plane $E_2\subset E$ with~$c_2 = 6$
evidently depends on~$Q$ (which is determined by~$E$).

\begin{Example}\label{ex: noninvolutivity}
Suppose that $E$ satisfies~$Q = 0$, which, by~\eqref{eq: QintermsofPandZ},
is the case for all of the admissible integral elements based at $(p,\eurox,\euroa)$
if $Z(p) = 0$.
In this case, it is clear that $\left[\alpha_E\right]+\left[\xi_E\right]Q
= \left[\alpha_E\right]$ takes values in skew-symmetric $3$-by-$3$
matrices and hence that, for every $2$-plane $E_2\subset E$,
one must have $H(E_2) = \ker\alpha_E$, so that $c_2=3$.
Thus, Cartan's inequality is strict,
and the integral element~$E$ is not Cartan-ordinary.

Note, though, that this does not imply that there are no solutions
to the original problem on domains containing~$p$ when $Z(p)=0$;
it's just that Cartan--K\"ahler cannot immediately be applied
in such situations. For example, note that,
when~$\Omega$ vanishes identically (equivalently, $Z$ vanishes identically),
then all of the admissible integral elements of~$\cI$
are contained in the integrable $6$-plane field $\alpha = 0$,
and, indeed, the general solution~$\omega$ is of the form~$\omega=\d \mbx$
where $\mbx\colon M\to\bbR^3$ is any immersion.
\end{Example}

For any $3$-by-$3$ matrix $Q$, define $A_Q\subset \eugl(3,\bbR)
= \Hom \big(\bbR^3,\bbR^3\big)$, the \emph{tableau} of~$Q$,
to be the span of the $3$-by-$3$ matrices
\begin{gather*}
[\mbx ]+ [\mby] Q
\end{gather*}
for $\mbx,\mby\in\mathbb{R}^3$.
The dimension of the vector space $A_Q$ lies between $3$ and $6$.

It is evident that the polar equations of flags
in a given admissible integral element $E$
defined by~\eqref{eq: E in PQterms} are governed
by the properties of the tableau~$A_Q$.

To simplify the study of~$A_Q$, it is useful to note that
it has a built-in equivariance:
For $R\in\SO(3)$, one has
\begin{gather*}
R\bigl([\mbx ]+ [\mby] Q\bigr)R^{-1} = R[\mbx]R^{-1}+R[\mby]R^{-1}\big(RQR^{-1}\big) = [R\mbx] + [R\mby] RQR^{-1}.
\end{gather*}
Hence,
\begin{gather*}
R A_Q R^{-1} = A_{RQR^{-1}}.
\end{gather*}
In particular, properties of~$A_Q$ such as its dimension, character sequence, and involutivity depend only on the equivalence class of the matrix~$Q$ under the action of conjugation by $\SO(3)$. Also, writing $Q = q I_3 + Q_0$ where $\tr(Q_0) = 0$, one has
\begin{gather*}
[ \mbx ]+ [\mby ] Q = [ \mbx + q \mby ]+ [\mby ] Q_0 .
\end{gather*}
Thus,
\begin{gather*}
A_Q = A_{Q_0}.
\end{gather*}

\begin{Proposition}\label{prop: nondegenerateA_Q}
The tableau~$A_Q\subset \eugl(3,\bbR) = \Hom \big(\bbR^3,\bbR^3\big)$
has dimension~$6$ and is involutive
with characters~$(s_1,s_2,s_3) = (3,3,0)$,
except when the trace-free part of~$Q$ is conjugate by $\SO(3)$
to a matrix of the form
\begin{gather}\label{eq: Q0normalizeddegenerate}
Q_0 \simeq \begin{pmatrix}-2x&0&0\\0&x+3r&3y\\0&-3y&x-3r\end{pmatrix},
\end{gather}
where $(x,y,r)$ are real numbers satisfying either $r^2 = x^2 + y^2$
or $r = y = 0$.
\end{Proposition}

\begin{proof}
The proof is basically a computation. The conjugation action of $\SO(3)$
on $3$-by-$3$ mat\-rices preserves the splitting of $\eugl(3,\bbR)$
into three pieces: The multiples of the identity (of dimension~$1$),
the subalgebra~$\euso(3)$ (of dimension $3$),
and the traceless symmetric matrices (of dimension $5$).
Moreover, as is well-known, a symmetric $3$-by-$3$ matrix can be
diagonalized by conjugating with an orthogonal matrix. Thus, one
is reduced to studying the case in which $Q_0$ is written in the form
\begin{gather}\label{eq: Q0normalized}
Q_0 = \begin{pmatrix}\phm q_1&\phm p_3&-p_2\\
 -p_3&\phm q_2&\phm p_1\\
 \phm p_2&-p_1&\phm q_3\end{pmatrix},
\end{gather}
where $q_1+q_2+q_3 = 0$.

It is now a straightforward (if somewhat tedious) matter
(which can be eased by MAPLE) to check that, when $A_{Q_0}$
has dimension less than $6$ (the maximum possible),
two of the $p_i$ must vanish. Thus, after conjugating by
a signed permutation matrix that lies in $\SO(3)$,
one can assume that $p_2=p_3=0$.
With this simplification, $A_{Q_0}$ is seen to have dimension
less than $6$ if and only if
\begin{gather*}
p_1 \big({p_1}^2 + 2{q_2}^2 + 5q_2q_3 + 2{q_3}^2\big) =
(q_2 - q_3) \big({p_1}^2 + 2{q_2}^2 + 5q_2q_3 + 2{q_3}^2\big) = 0.
\end{gather*}
Thus, either ${p_1}^2 + 2{q_2}^2 + 5q_2q_3 + 2{q_3}^2=0$ or $p_1=q_2-q_3=0$.
Making the necessary changes of basis,
these two cases give the two non-involutive normal forms
in~\eqref{eq: Q0normalizeddegenerate}.

It remains to show that, when $A_Q$ has dimension~$6$,
it actually is involutive with the stated characters~$(s_1,s_2,s_3)=(3,3,0)$.
To do this, return to the general normal form~\eqref{eq: Q0normalized},
and assume that $A_Q$ has dimension~$6$.
Because $A_Q$ has codimension~$3$ in~$\eugl(3,\bbR)$,
it will be involutive with characters~$(s_1,s_2,s_3)=(3,3,0)$
if and only if it has a non-characteristic covector.
Now, the condition that a covector~$\mbz^* = (z_1,z_2,z_3)\in \big(\bbR^3\big)^*$
be characteristic for~$A_Q$ is the condition
that the $3$-dimensional vector space of rank~$1$ matrices
of the form~$\mbx \mbz^*$ (where $\mbx\in\bbR^3$ and $\mbz^* = (z_1,z_2,z_3)$
is regarded as a row vector) have a nontrivial intersection
with~$A_Q$ in~$\eugl(3,\mathbb{R})$.
The condition that a rank $1$ matrix~$\mathbf{r} = \mathbf{x}\mathbf{z}^*$
lie in the $6$-dimensional subspace $A_Q$ of the $9$-dimensional space
$\frak{gl}(3,\mathbb{R})$ can be expressed as 3 homogeneous linear equations in~$\mathbf{r}$,
i.e., $3$ homogeneous equations bilinear in the components
of~$\mathbf{x}$ and $\mathbf{z}^*$. Regarding $\mathbf{z}^*\not=0$ as given,
this becomes a system of three linear equations for the components
of~$\mathbf{x}$ whose coefficient matrix~$C_Q(\mathbf{z}^*)$ is $3$-by-$3$
with entries that are linear in the components of~$\mathbf{z}^*$.
This system will have a nonzero solution~$\mathbf{x}$
if and only if $\det\bigl(C_Q(\mathbf{z}^*)\bigr) = 0$.
In terms of the coefficients~$p_i$ and~$q_i$ of~$Q_0$,
this determinant vanishing can be written
as a homogeneous cubic polynomial equation
\begin{gather*}
0 = \sum_{ijk} c_{ijk}(p,q) z_i z_j z_k = c_Q(\mbz^*).
\end{gather*}
One then finds (again by a somewhat tedious calculation
that is eased by MAPLE) that this equation holds identically in~$\mbz^*$
(i.e., that all of the~$c_{ijk}(p,q)$ vanish) if and only if~$Q_0$
is equivalent to a matrix of the form~\eqref{eq: Q0normalizeddegenerate}
subject to either of the two conditions~$r=y=0$ or $r^2 = x^2+y^2$.

Thus, except when $Q_0$ is orthogonally equivalent to such matrices,
$A_Q$ has dimension~$6$ and there exists a non-characteristic covector~$\mbz^*$ for~$A_Q$. As already explained, this implies that~$A_Q$ is
involutive, with the claimed Cartan characters.
\end{proof}

\begin{Remark}The $\SO(3)$-orbits of the matrices~$Q$
whose trace-free part~$Q_0$
is of the form~\eqref{eq: Q0normalizeddegenerate}
with~$r=y=0$ forms a closed cone of dimension~$4$
in the ($9$-dimensional) space~$\eugl(3,\mathbb{R})$
of $3$-by-$3$ matrices.
Meanwhile, the $\SO(3)$-orbits of the matrices~$Q$
whose trace-free part~$Q_0$
is of the form~\eqref{eq: Q0normalizeddegenerate}
with~$r^2=x^2+y^2$ forms a closed cone of dimension~$6$
in~$\eugl(3,\mathbb{R})$.

Consequently, the set consisting of those~$Q$
for which $A_Q$ is involutive is an open dense set
in the space~$\eugl(3,\mathbb{R})$.
\end{Remark}

\begin{Remark}\label{rem: Q-hyperbolicity}
It does not appear to be easy to determine the condition on~$Q$
that the real cubic curve~$c_Q(\mbz^*) = 0$ be a smooth,
irreducible cubic with two circuits. This is what one would need
in order to have a chance of showing that the (linearized) equation
were symmetric hyperbolic, which would be a key step in proving
solvability of the original problem in the smooth category.
\end{Remark}

\begin{Corollary}\label{cor: ECartan-ordinary}
If $E\in \cV_3\bigl(\cI,(\xi,\eta)\bigr)$
is defined by equations~\eqref{eq: E in PQterms},
then $E$ is Cartan-regular
if and only if $Q_0 = Q - \tfrac13\tr(Q)I_3$
is not orthogonally equivalent
to a matrix of the form~\eqref{eq: Q0normalizeddegenerate},
where either $r = y =0$ or $r^2 = x^2+y^2$.
\end{Corollary}

\begin{proof}
Everything is clear from Proposition~\ref{prop: nondegenerateA_Q},
except possibly the assertion of Cartan-\emph{regularity}.
However, because the characters are~$(s_1,s_2,s_3)=(3,3,0)$,
when $Q$ avoids the two `degenerate' cones, it follows that,
when $E\in\cV_3\bigl(\cI,(\xi,\eta)\bigr)$ has the property
that its~$A_Q$ is involutive, then, for any non-characteristic
$2$-plane~$E_2\subset E$, we must have $H(E_2) = E$, and hence
$H(E)=E$, so that $E$ must be not only Cartan-ordinary, but
also Cartan-regular.
\end{proof}

\section{Involutivity}\label{sec: Involutivity}
Finally, we collect all of this information together,
yielding our main result:

\begin{Theorem}\label{nondegenerate-involutive-theorem}
Let $\Omega$ be a real-analytic closed $2$-form on a $3$-manifold~$M$ with values in~$\bbR^3$, and suppose that there is no nonzero vector~$\mbv\in T_pM$
such that $\mbv\lhk\Omega = 0$. Then there is an open $p$-neighborhood~$U\subset M$ on which there exists an $\bbR^3$-valued coframing~$\omega\colon TU\to\bbR^3$
such that $\d\omega = \Omega_U$ and such that the metric~$g = {}^t\omega\circ \omega$ is flat. Moreover, the space of such coframings~$\omega$ depends locally on $3$ functions of $2$ variables.
\end{Theorem}

\begin{proof}Keeping the established notation, it suffices to show that, if~$Z(p)$ has rank at least~$2$, then there exists a $P\in\GL(3,\bbR)$ such that, when $Q$ is defined by~\eqref{eq: QintermsofPandZ}, the tableau~$A_Q$ is involutive.

Now, by the hypothesis that there is no nonzero vector $\mbv\in T_pM$
such that $\mbv\lhk\Omega=0$, the rank of $Z(p)$ is either $2$ or $3$.
When the rank of $Z(p)$ is~$3$,
as $P$ varies over~$\GL(3,\bbR)$,
the matrix~$Q$ varies over an open subset of~$\GL(3,\bbR)$,
and it is clear that, for the generic choice of~$P$,
the corresponding $Q_0$ will not be $\SO(3)$-equivalent to anything
in the two `degenerate' cones defined by~\eqref{eq: Q0normalizeddegenerate}
with either $r=y=0$ or $r^2=x^2+y^2$.

When the rank of $Z(p)$ is $2$, we can assume, after an $\SO(3)$ rotation,
that the bottom row of~$Z(p)$ vanishes and that the first two rows of~$Z(p)$
are linearly independent. It then follows that $P/(\det P)\, {}^t\!Z(p)$
has its last column equal to zero, but that, as $P$ varies, the first
two columns of $P/(\det P)\, {}^t\!Z(p)$ range over all linearly independent
pairs of column vectors. Now explicitly computing the polynomial $c_Q(\mbz^*)$
for the corresponding matrix~$Q$ shows that $c_Q(\mbz^*)$ does not vanish
identically on the set of such matrices, hence it is possible to choose $P$
so that $c_Q(\mbz^*)$ does not vanish identically, and the corresponding $A_Q$
is then involutive, implying that the corresponding admissible integral
element~$E$ is Cartan-ordinary.

In either case, there exist Cartan-ordinary admissible integral elements
of~$\cI$ based at~$p$, so the Cartan--K\"ahler theorem applies, showing
that there exist admissible integral manifolds of~$\cI$ passing through
any point~$(p,\mbx,\mba)\in X^9$, and hence,
by Proposition~\ref{prop: formulation_as_EDS},
the original problem is solvable in an open neighborhood of~$p$.
Moreover, since the last nonzero Cartan character
of a generic integral flag is~$s_2 = 3$,
the space of solutions~$\omega$
depends locally on $3$ functions of $2$ variables,
in the sense of Cartan.
\end{proof}

\section{The rank 1 case}\label{sec: Zrank1}
If the rank of $Z(p)$ is either 0 or 1,
then, for all values of $Q$ as defined in~\eqref{eq: QintermsofPandZ}
with $P$ invertible, the tableau $A_Q$ fails to be involutive,
so the Cartan--K\"ahler theorem cannot be applied to prove local solvability.

However, as noted in Example~\ref{ex: noninvolutivity},
this does not necessarily preclude
the existence of integral manifolds of~$\cI$ in a neighborhood of~$p$.
Indeed, when $Z$ vanishes identically on a neighborhood of $p \in M$,
the general solution $\omega = \d \mbx$ (where $\mbx\colon M \to \bbR^3$
is an arbitrary immersion)
depends locally on 3 functions of 3 variables;
so there are actually {\em more} integral manifolds
in this case than in the case in which~$Z(p)$ has rank~$2$ or~$3$.

Nevertheless, as the following example demonstrates, even local solvability is not guaranteed in general.

\begin{Example}\label{ex: nonsolvability}
Set $\Omega = (\Omega^i) = (\Upsilon,0,0)$, where
\begin{gather}\label{define-Upsilon}
\Upsilon = u^1 \d u^2\wedge\d u^3 + u^2 \d u^3\wedge\d u^1 - 2u^3 \d u^1\wedge\d u^2.
\end{gather}
(Note that in this case, the matrix $Z$ has rank~$1$
everywhere except at the origin, where the rank is~$0$.)
We will show that there is no coframing $\omega = (\omega^i)$ on any neighborhood of $u = (u^i) = (0,0,0)$ such that the metric $g = {}^t\omega\circ \omega$ is flat.
In fact, we will show, more generally, that if $\omega$ is any coframing on $M$ such that $\d\omega^2=\d\omega^3=0$
and the metric $g = {}^t\omega\circ \omega$ is flat,
then we must have $\omega^1 \wedge \d\omega^1 = 0$.

Meanwhile, $\Upsilon$ defined as in \eqref{define-Upsilon} has no nonvanishing factor on any neighborhood of $u = (u^i) = (0,0,0)$.
In order to see this, suppose that $\Upsilon\wedge \beta = 0$,
where $\beta = b_1 \d u^1 + b_2 \d u^2 + b_3 \d u^3$.
Then
\[ u^1b_1 + u^2b_2-2u^3b_3 = 0. \]
This implies, for example,
that $u^3b_3$ must vanish on the line $u^1=u^2=0$ and hence
that $b_3$ must also vanish there. In particular, $b_3$
must vanish at the origin $u^i=0$. Similarly, $b_1$ and $b_2$
must also vanish at the origin. Thus, $\beta$ must vanish
at the origin.

To establish the general claim, let $\omega$ be a coframing on~$M^3$ such that $\d\omega^2=\d\omega^3=0$
and the metric $g = {}^t\omega\circ \omega$ is flat.
Writing
\begin{gather*}
\d\begin{pmatrix}\omega^1\\ \omega^2\\ \omega^3\end{pmatrix}
=-\begin{pmatrix}\phm0&\phm\phi^3&-\phi^2\\-\phi^3&\phm0&\phm\phi^1\\
\phm\phi^2&-\phi^1&\phm0\end{pmatrix}
\wedge \begin{pmatrix}\omega^1\\ \omega^2\\ \omega^3\end{pmatrix}
= \begin{pmatrix}\d\omega^1\\0\\0\end{pmatrix},
\end{gather*}
we see, from the vanishing of $\d\omega^2$ and $\d\omega^3$, that there must exist functions~$a^1$, $a^2$, and $a^3$
such that
\begin{gather*}
\phi^1 = a^1 \omega^1,\qquad
\phi^2 = a^2 \omega^1 - a^1 \omega^2,\qquad
\phi^3 = a^3 \omega^1 - a^1 \omega^3.
\end{gather*}
Consequently, we must have
\begin{gather*}
\d\omega^1 = -2a^1 \omega^2\wedge\omega^3 - a^2 \omega^3\wedge\omega^1 - a^3 \omega^1\wedge\omega^2.
\end{gather*}
Now, the flatness of the metric~$g$ is equivalent to the equations
\begin{gather*}
\d\phi^1-\phi^2\wedge\phi^3 = \d\phi^2-\phi^3\wedge\phi^1 = \d\phi^3-\phi^1\wedge\phi^2 = 0.
\end{gather*}
However, from the above equations, we see that
\begin{gather*}
0 = \d\phi^1-\phi^2\wedge\phi^3
= \d a^1\wedge\omega^1 -3\big(a^1\big)^2 \omega^2\wedge\omega^3
 -2a^1a^2 \omega^3\wedge\omega^1-2a^1a^3 \omega^1\wedge\omega^2.
\end{gather*}
Wedging both ends of this equation with $\omega^1$ yields
$-3\big(a^1\big)^2 \omega^1\wedge\omega^2\wedge\omega^3 = 0$. Hence $a^1 = 0$,
and we have
\begin{gather*}
\d\omega^1 = \omega^1\wedge\big(a^2 \omega^3-a^3 \omega^2\big).
\end{gather*}
In particular, $\omega^1\wedge\d\omega^1 = 0$, as claimed.

It is worthwhile to carry these calculations with the
coframing~$\omega$ a little further. Since~$a^1 = 0$,
we see that $\phi^1 = 0$, and the condition for flatness
reduces to $\d\phi^2 = \d\phi^3 = 0$.

Let us assume that $M$ is connected and simply-connected.
Fix a point $p\in M$ and write $\omega^2 = \d u^2$ and $\omega^3 = \d u^3$
for unique functions $u^2$ and $u^3$ that vanish at~$p$.
Since $\omega^1\wedge\d\omega^1 = 0$, it follows from the Frobenius Theorem that there exists
an open $p$-neighborhood $U\subset M$ on which
there exists a function $u^1$ vanishing at $p$ such that
$\omega^1 = f \d u^1$ for some nonvanishing function $f$ on~$U$.
Restricting to a smaller $p$-neighborhood if necessary, we can
arrange that $\mbu = \big(u^1,u^2,u^3\big)\colon U\to\bbR^3$ be a rectangular coordinate chart.
Now, computation yields
\begin{gather*}
\phi^1 = 0,\qquad \phi^2 = -\frac{\partial f}{\partial u^3} \d u^1,
\qquad \phi^3 = \frac{\partial f}{\partial u^2} \d u^1.
\end{gather*}
The remaining flatness conditions $\d\phi^2 = \d\phi^3 = 0$
then are equivalent to
\begin{gather*}
\frac{\partial^2 f}{\big(\partial u^2\big)^2}
= \frac{\partial^2 f}{\partial u^2\partial u^3}
= \frac{\partial^2 f}{\big(\partial u^3\big)^2} = 0.
\end{gather*}
Consequently, $f = f\big(u^1,u^2,u^3\big)$ is linear in $u^2$ and $u^3$,
so it can be written in the form $f= g_1\big(u^1\big) + g_2\big(u^1\big)u^2+g_3\big(u^1\big)u^3$
for some functions $g_1$, $g_2$, $g_3$. Since $f$ does not vanish
on $u^2 = u^3 = 0$, by changing coordinates in $u^1$, we can arrange
that $g_1\big(u^1\big) = 1$. Thus, the coframing takes the form
\begin{gather*}
\omega = \big( \big(1 + g_2\big(u^1\big)u^2 + g_3\big(u^1\big)u^3\big) \d u^1, \d u^2, \d u^3\big),
\end{gather*}
where the $p$-centered coordinates~$u^i$ are unique. Conversely,
for any two functions $g_2$ and $g_3$ on an interval containing $0\in\bbR$,
the above coframing has the property that $\d\omega^2=\d\omega^3=0$
while the metric $g = {}^t\omega\circ\omega$ is flat. Finally,
note that $\d\omega^1$ is nonvanishing at $\mbu=0$ if and only if $g_2(0)$
and $g_3(0)$ are not both zero.
\end{Example}

In light of Example~\ref{ex: nonsolvability}, it is clear that some assumptions will be required in order to ensure that local solutions exist. First, in order to avoid a singularity of the type in Example~\ref{ex: nonsolvability}, where~$Z$ vanishes at a single point, we will assume that $Z$ has constant rank 1 in some neighborhood~$U$ of~$p \in M$. This assumption is equivalent to the assumption that the 2-forms $\Omega^1$,~$\Omega^2$,~$\Omega^3$ are scalar multiples of each other and do not simultaneously vanish.

\subsection{Formulation as an exterior differential system}
We will take the following approach: Rather than assuming that $Z$ is specified in advance, we will seek to characterize functions $\mbx\colon U \to \bbR^3$, $\mba\colon U \to \SO(3)$ such that the components $\big(\omega^1, \omega^2, \omega^3\big)$ of the $\bbR^3$-valued 1-form $\omega = \mba^{-1} \d\mbx$ form a local coframing on~$U$ with the property that the 2-forms $\big(\d\omega^1, \d \omega^2, \d\omega^3\big)$ are pairwise linearly dependent and do not vanish simultaneously. Since this property is invariant under reparametrizations of the domain~$U$, it suffices to characterize 3-dimensional submanifolds $N^3 \subset \bbR^3 \times \SO(3)$ that are graphs of functions with this property. In practice, this means that the coordinates $\eurox = \big(x^1, x^2, x^3\big)$ on the open subset \smash{$V = \mbx(U) \subset \bbR^3$} may be regarded as the independent variables on any such submanifold $N^3$, and the map $\mba\colon U \to \SO(3)$ may be regarded as a function~$\mba(\eurox)$, i.e., as a~map $\mba\colon V \to \SO(3)$.
As in Section~\ref{EDS-sec}, we define the $\mathbb{R}^3$-valued $1$-forms $\xi$ and $\alpha$ on $\bbR^3 \times \SO(3)$ by equation \eqref{basis-forms-on-R3xSO3}; we will regard the 1-forms $\big(\omega^1, \omega^2, \omega^3\big)$ as the pullbacks to $V$ of the 1-forms $\big(\xi^1, \xi^2, \xi^3\big)$ on $\bbR^3 \times \SO(3)$.

Any 3-dimensional submanifold $N^3$ of the desired form must have the property that the 1-forms $\big(\xi^1, \xi^2, \xi^3\big)$ restrict to be linearly independent on $N^3$ and hence form a basis for the linearly independent 1-forms on $N^3$. Thus the restrictions of the 1-forms $\big(\alpha^1, \alpha^2, \alpha^3\big)$ to $N^3$ may be written as
\[ \alpha^i = y^i_j \xi^j \]
for some functions $y^i_j$ on $N^3$. Then from the structure equations \eqref{structure-eqs-on-R3xSO3}, we have
\begin{gather}
\begin{pmatrix} \d\xi^1 \\ \d\xi^2 \\ \d\xi^3 \end{pmatrix} =
 -\begin{pmatrix}
-(y^2_2 + y^3_3) & y^2_1 & y^3_1 \\
y^1_2 & -(y^3_3 + y^1_1) & y^3_2 \\
y^1_3 & y^2_3 & -(y^1_1 + y^2_2) \end{pmatrix}
\begin{pmatrix} \xi^2 \wedge \xi^3 \\ \xi^3 \wedge \xi^1 \\ \xi^1 \wedge \xi^2 \end{pmatrix} \nonumber\\
\hphantom{\begin{pmatrix} \d\xi^1 \\ \d\xi^2 \\ \d\xi^3 \end{pmatrix}}{}
 = -\big({}^t\hskip-1pt\big(y^i_j\big) - \tr \big(\big(y^i_j\big)\big) I_3 \big) \begin{pmatrix} \xi^2 \wedge \xi^3 \\ \xi^3 \wedge \xi^1 \\ \xi^1 \wedge \xi^2 \end{pmatrix} .
\label{matrix-for-structure-eqs}
\end{gather}
The condition that the 2-forms $\big(\d\omega^1, \d \omega^2, \d\omega^3\big)$ are pairwise linearly dependent and do not vanish simultaneously on~$U$ is equivalent to the condition that the same is true for the 2-forms $\big(\d\xi^1, \d \xi^2, \d\xi^3\big)$ on $N^3$, and hence that the matrix in equation~\eqref{matrix-for-structure-eqs} has rank 1 on~$N^3$. This, in turn, is equivalent to the condition that
\[ \big(y^i_j\big) = \lambda I_3 + M \]
for some matrix $M$ of constant rank 1 on $N^3$, with $\lambda = -\tfrac{1}{2} (\tr M)$.

\begin{Remark}\label{lambda-remark}The function $\lambda$ has the following interpretation: equations \eqref{matrix-for-structure-eqs} imply that on any integral manifold, the $1$-forms $\big(\omega^1, \omega^2, \omega^3\big)$ satisfy the equation
\begin{gather*}
\omega^1 \wedge \d \omega^1 + \omega^2 \wedge \d \omega^2 + \omega^3 \wedge \d \omega^3 = -2\lambda \omega^1 \wedge \omega^2 \wedge \omega^3.
\end{gather*}
As we will see, the cases where $\lambda=0$ and $\lambda \neq 0$ behave quite differently.
\end{Remark}

Since the matrix $M$ has rank 1 on $N^3$, it can be written as
\[ M = \mbv {}^t\hskip-1pt\mbw = \begin{pmatrix} v^1 \\ v^2 \\ v^3 \end{pmatrix}
\begin{pmatrix} w_1 & w_2 & w_3 \end{pmatrix} \]
for some nonvanishing $\bbR^3$-valued functions $\mbv$, $\mbw$ on $N^3$ that are determined up to a scaling transformation
\[ \mbv \to r \mbv, \qquad \mbw \to r^{-1} \mbw. \]
Without loss of generality, we may take advantage of this scaling transformation to assume that~$\mbv$ is a unit vector at each point of~$N^3$. Then, since $\tr (M) = -2\lambda$, we can choose an oriented, orthonormal frame field $(\mbf_1, \mbf_2, \mbf_3)$ along $N^3$ with the property that
\[ \mbv = \mbf_1, \qquad \mbw = -2 \lambda \mbf_1 + \mu \mbf_2 \]
for some real-valued function $\mu$ on $N^3$.

Let $\mbf \in \SO(3)$ denote the orthogonal matrix
\[ \mbf = [ \begin{matrix} \mbf_1 & \mbf_2 & \mbf_3\end{matrix}]. \]
Since we have $\mbf \, {}^t\hskip-1pt\mbf = I_3$, we can write the matrix $\big[y^i_j\big]$ as
\begin{gather*}
\big[y^i_j\big] = \lambda I_3 + M = \lambda \big(\mbf I_3 \, {}^t\hskip-1pt\mbf\big) + \mbf_1 \big({-}2\lambda\, {}^t\hskip-1pt \mbf_1 + \mu \, {}^t\hskip-1pt \mbf_2\big) \\
\hphantom{\big[y^i_j\big]}{} = \mbf \left(\lambda I_3 + \begin{pmatrix} -2\lambda & \mu & 0 \\ 0 & 0 & 0 \\ 0 & 0 & 0 \end{pmatrix} \right) {}^t\hskip-1pt\mbf
 = \mbf \begin{pmatrix} -\lambda & \mu & 0 \\ 0 & \lambda & 0 \\ 0 & 0 & \lambda \end{pmatrix} {}^t\hskip-1pt\mbf.
\end{gather*}

This discussion suggests that we introduce the following exterior differential system: Let $X$ denote the $11$-dimensional manifold
\[ X = \bbR^3 \times \SO(3) \times \SO(3) \times \bbR^2, \]
with coordinates $(\eurox, \euroa, \eurof, (\lambda, \mu))$. We may take the 1-forms
$\big(\xi^i, \alpha^i, \varphi^i, \d\lambda, \d\mu\big)$
as a basis for the 1-forms on $X$, where the 1-forms $\big(\varphi^1, \varphi^2, \varphi^3\big)$ are the standard Maurer--Cartan forms on the second copy of $\SO(3)$ and so are defined by the equation
\[ [\varphi] = \begin{pmatrix} 0 & \varphi^3 & -\varphi^2 \\ -\varphi^3 & 0 & \varphi^1 \\ \varphi^2 & -\varphi^1 & 0
\end{pmatrix} = \eurof^{-1} \d\eurof. \]
Let $\cI$ be the exterior differential system on $X$ that is generated by the three $1$-forms $\big(\theta^1, \theta^2, \theta^3\big)$, where
\[ \begin{pmatrix} \theta^1 \\ \theta^2 \\ \theta^3 \end{pmatrix} = \begin{pmatrix} \alpha^1 \\ \alpha^2 \\ \alpha^3 \end{pmatrix} - \eurof \begin{pmatrix} -\lambda & \mu & 0 \\ 0 & \lambda & 0 \\ 0 & 0 & \lambda \end{pmatrix} {}^t\hskip-1pt\eurof
\begin{pmatrix} \xi^1 \\ \xi^2 \\ \xi^3 \end{pmatrix}. \]

\begin{Proposition}\label{prop: formulation_as_EDS-rank1}
If $N^3\subset X$ is an integral manifold of~$\cI$ to which $\xi$ pulls back to be a coframing, then each point of~$N^3$ has an open neighborhood that can be written as a graph
\[ \bigl\{\bigl(\eurox, \mba(\eurox), \mbf(\eurox), \lambda(\eurox), \mu(\eurox) \bigr)\, \vrule\, \eurox \in V \bigr\} \subset X \]
for some open set $V \subset \bbR^3$ and smooth maps $\mba, \mbf\colon V \to \SO(3)$ and $\lambda, \mu\colon V \to \bbR$. Moreover, on~$V$, the coframing $\xi = \mba^{-1} \d \eurox$ satisfies the structure equations
\begin{gather*} 
\begin{pmatrix} \d\xi^1 \\ \d\xi^2 \\ \d\xi^3 \end{pmatrix} = \mbf \begin{pmatrix} -2\lambda & 0 & 0 \\ \mu & 0 & 0 \\ 0 & 0 & 0 \end{pmatrix} {}^t\hskip-1pt \mbf
\begin{pmatrix} \xi^2 \wedge \xi^3 \\ \xi^3 \wedge \xi^1 \\ \xi^1 \wedge \xi^2 \end{pmatrix},
\end{gather*}
and the metric~$g = {}^t\xi\circ \xi = {}^t\d \eurox\circ \d \eurox$ is flat.

Conversely, if $V\subset \bbR^3$ is a simply-connected open subset on which there exists a coframing~$\xi\colon TV\to\bbR^3$ satisfying $($i$)$ the $2$-forms $\d\xi^i$ are pairwise linearly dependent and nowhere simultaneously vanishing, and $($ii$)$ the metric~$g={}^t\xi\circ \xi$ is flat, then there exist mappings~$\mba, \mbf\colon V \to \SO(3)$ and $\lambda, \mu\colon V \to \bbR$ such that $\xi = \mba^{-1} \d \eurox$. Moreover, the immersion $\iota\colon V\to X$ defined by $\iota(\eurox) = \bigl(\eurox, \mba(\eurox), \mbf(\eurox), \lambda(\eurox), \mu(\eurox) \bigr)$ is an integral manifold of~$\cI$ that pulls~$\xi$ back to be a coframing of~$V$.
\end{Proposition}

\begin{proof}The proof is similar to that of Proposition~\ref{prop: formulation_as_EDS}.
\end{proof}

It turns out that the calculations involved in the analysis of this exterior differential system are much simpler if we introduce the 1-forms
\[ \begin{pmatrix} \chi^1 \\ \chi^2 \\ \chi^3 \end{pmatrix} = {}^t\hskip-1pt\eurof \begin{pmatrix} \xi^1 \\ \xi^2 \\ \xi^3 \end{pmatrix}\]
on $X$ and replace $\big(\xi^1, \xi^2, \xi^3\big)$ by the equivalent expressions
\[ \begin{pmatrix} \xi^1 \\ \xi^2 \\ \xi^3 \end{pmatrix} = \eurof \begin{pmatrix} \chi^1 \\ \chi^2 \\ \chi^3 \end{pmatrix}. \]
It is straightforward to show that the 1-forms $\big(\chi^1, \chi^2, \chi^3\big)$ satisfy the structure equations
\begin{gather*}
\begin{pmatrix} \d\chi^1 \\ \d\chi^2 \\ \d\chi^3 \end{pmatrix} =
-\left( [\varphi] + [{}^t\hskip-1pt\eurof \alpha] \right) \wedge \begin{pmatrix} \chi^1 \\ \chi^2 \\ \chi^3 \end{pmatrix} \\
\hphantom{\begin{pmatrix} \d\chi^1 \\ \d\chi^2 \\ \d\chi^3 \end{pmatrix}}{}
 \equiv -\begin{pmatrix} 0 & \varphi^3 & -\varphi^2 \\ -\varphi^3 & 0 & \varphi^1 \\ \varphi^2 & -\varphi^1 & 0
\end{pmatrix} \wedge \begin{pmatrix} \chi^1 \\ \chi^2 \\ \chi^3 \end{pmatrix} + \begin{pmatrix} 2\lambda \\ -\mu \\ 0 \end{pmatrix} \chi^2 \wedge \chi^3 \mod{\cI},
\end{gather*}
and we can now write the generators of $\cI$ as
\begin{gather} \label{rank1-generators-for-I}
\begin{pmatrix} \theta^1 \\ \theta^2 \\ \theta^3 \end{pmatrix} = \begin{pmatrix} \alpha^1 \\ \alpha^2 \\ \alpha^3 \end{pmatrix} - \eurof \begin{pmatrix} -\lambda & \mu & 0 \\ 0 & \lambda & 0 \\ 0 & 0 & \lambda \end{pmatrix}
\begin{pmatrix} \chi^1 \\ \chi^2 \\ \chi^3 \end{pmatrix}.
\end{gather}
The exterior differential system $\cI$ is generated algebraically by the $1$-forms $\big(\theta^1, \theta^2, \theta^3\big)$ and their exterior derivatives $\big(\d\theta^1, \d\theta^2, \d\theta^3\big)$.

The value of $\lambda$ on any particular integral manifold $N^3$ plays a crucial role here. If $\lambda=0$ on $N^3$, then the 1-forms $\big(\alpha^1, \alpha^2, \alpha^3\big)$ are all multiples of the single $1$-form $\chi^2$, and therefore the corresponding map $\mba\colon V \to \SO(3)$ has rank 1; in particular, the image of~$\mba$ is a curve in~$\SO(3)$. On the other hand, if $\lambda \neq 0$ on $N^3$, then the 1-forms $\big(\alpha^1, \alpha^2, \alpha^3\big)$ are linearly independent, and therefore the corresponding map $\mba\colon V \to \SO(3)$ has rank~3 and is a local diffeomorphism from~$V$ onto an open subset of~$\SO(3)$. Due to these different behaviors, the analysis of this exterior differential system varies considerably depending on whether or not $\lambda$ vanishes, and so we will consider these cases separately.

\subsection[The case $\lambda = 0$]{The case $\boldsymbol{\lambda = 0}$}\label{lambda-zero-sec}

Consider the restriction $\bar{\cI}$ of $\cI$ to the codimension 1 submanifold $\bar{X}$ of $X$ defined by the equation $\lambda=0$.
The rank 1 condition implies that any integral manifold must be contained in the open set where $\mu \neq 0$, and the expressions~\eqref{rank1-generators-for-I} reduce to
\begin{gather}\label{rank1-lambda-zero-thetas}
\begin{pmatrix} \theta^1 \\ \theta^2 \\ \theta^3 \end{pmatrix} =
\begin{pmatrix} \alpha^1 \\ \alpha^2 \\ \alpha^3 \end{pmatrix} - \eurof
\begin{pmatrix} 0 & \mu & 0 \\ 0 & 0 & 0 \\ 0 & 0 & 0 \end{pmatrix}
\begin{pmatrix} \chi^1 \\ \chi^2 \\ \chi^3 \end{pmatrix}.
\end{gather}
Differentiating equations~\eqref{rank1-lambda-zero-thetas}, reducing modulo $\big(\theta^1, \theta^2, \theta^3\big)$, and multiplying on the left by ${}^t\hskip-1pt \eurof$ yields
\begin{gather}\label{rank1-lambda-zero-dthetas}
 {}^t\hskip-1pt \eurof \begin{pmatrix} \d\theta^1 \\ \d\theta^2 \\ \d\theta^3 \end{pmatrix} \equiv
- \begin{pmatrix} \pi_1 & \pi_2 & \pi_3 \\ 0 & -\pi_1 & 0 \\ 0 & \pi_4 & 0 \end{pmatrix} \wedge \begin{pmatrix} \chi^1 \\ \chi^2 \\ \chi^3 \end{pmatrix} \mod{\theta^1, \theta^2, \theta^3},
\end{gather}
where
\[
\pi_1 = \mu \varphi^3, \qquad \pi_2 = \d\mu + \mu^2 \chi^3, \qquad \pi_3 = -\mu \varphi^1, \qquad \pi_4 = \mu \varphi^2.
\]
The tableau matrix in equation \eqref{rank1-lambda-zero-dthetas} has Cartan characters $s_1 = 3$, $s_2 = 1$, $s_3=0$, and the space of integral elements at each point of $\bar{X}$ is 5-dimensional, parametrized by
\begin{gather*}
\pi_1 = p_1 \chi^2, \qquad
\pi_2 = p_1 \chi^1 + p_2 \chi^2 + p_3 \chi^3, \qquad
\pi_3 = p_3 \chi^2 + p_4 \chi^3, \qquad
\pi_4 = p_5 \chi^2,
\end{gather*}
with $p_1, p_2, p_3, p_4, p_5 \in \bbR$. Since $s_1 + 2 s_2 + 3 s_3 = 5$, the system $\bar{\cI}$ is involutive, with integral manifolds locally parametrized by~1 function of 2 variables.

As a result of this computation and Remark~\ref{lambda-remark}, we have the following theorem.

\begin{Theorem}\label{lambda-zero-function-count-theorem}
The space of local orthonormal coframings $\big(\omega^1, \omega^2, \omega^3\big)$ on an open subset of $\bbR^3$ whose exterior derivatives $\big(\d\omega^1, \d\omega^2, \d\omega^3\big)$ are pairwise linearly dependent and do not simultaneously vanish and satisfy the additional property that
\[ \omega^1 \wedge \d \omega^1 + \omega^2 \wedge \d \omega^2 + \omega^3 \wedge \d \omega^3 = 0 \]
is locally parametrized by $1$ function of $2$ variables.
\end{Theorem}

This function count suggests that, if the rank 1 matrix $Z$ on $M$ is specified in advance, local solutions are likely to exist for arbitrary, generic choices of $Z$. More specifically, by Darboux's Theorem, the rank $1$ condition implies that we can find local coordinates $\big(u^1, u^2, u^3\big)$ on some neighborhood $U$ of any point $p \in M$ such that
\[ \Omega = \mbz\big(u^1, u^2\big) \d u^1 \wedge \d u^2 \]
for some smooth, nonvanishing $\bbR^3$-valued function $\mbz\big(u^1, u^2\big)$. Moreover, by local coordinate transformations of the form $\big(u^1, u^2, u^3\big) \to \big(\tilde{u}^1\big(u^1, u^2\big), \tilde{u}^2\big(u^1, u^2\big), u^3\big)$, we might expect that we could normalize~2 of the~3 functions $z^i\big(u^1, u^2\big)$. For example, if $\d\big(z^1/z^2\big)(p) \neq 0$, then we could choose the functions $\tilde{u}^1$, $\tilde{u}^2$ in a neighborhod of $p$ such that $z^1\big(\tilde{u}^1, \tilde{u}^2\big) = 1$ and $z^2\big(\tilde{u}^1, \tilde{u}^2\big) = \tilde{u}^1$. Then the vector $\Omega$ is characterized by the remaining single function of 2 variables $z^3\big(\tilde{u}^1, \tilde{u}^2\big)$. Since this function account agrees with that for the space of integral manifolds of~$\cI$, one might hope that generic choices for the function $\mbz\big(u^1, u^2\big)$ would admit solutions.

In Section~\ref{explicit-solns-sec}, we will show that this is in fact the case; specifically, a mild nondege\-ne\-ra\-cy condition on the function $\mbz\big(u^1, u^2\big)$ suffices to guarantee the existence of solutions. (See Theorem~\ref{big-theorem} below for details.)

\subsection[The case $\lambda \neq 0$]{The case $\boldsymbol{\lambda \neq 0}$}

Now consider integral manifolds of $\cI$ contained in the open subset of $X$ where $\lambda \neq 0$. First we show that there are no integral manifolds on which $\mu=0$. To this end, suppose for the sake of contradiction that $\mu=0$ on some integral manifold $N^3$. Then the expressions~\eqref{rank1-generators-for-I} reduce to
\[
\begin{pmatrix} \theta^1 \\ \theta^2 \\ \theta^3 \end{pmatrix} = \begin{pmatrix} \alpha^1 \\ \alpha^2 \\ \alpha^3 \end{pmatrix} - \eurof \begin{pmatrix} -\lambda & 0 & 0 \\ 0 & \lambda & 0 \\ 0 & 0 & \lambda \end{pmatrix}
\begin{pmatrix} \chi^1 \\ \chi^2 \\ \chi^3 \end{pmatrix}.
\]
Differentiating these equations, reducing modulo $\big(\theta^1, \theta^2, \theta^3\big)$, and multiplying on the left by ${}^t\hskip-1pt \eurof$ yields
\[
 {}^t\hskip-1pt \eurof \begin{pmatrix} \d\theta^1 \\ \d\theta^2 \\ \d\theta^3 \end{pmatrix} \equiv
-\begin{pmatrix} \pi_1 & \pi_2 & \pi_3 \\ \pi_2 & -\pi_1 & 0 \\ \pi_3 & 0 & -\pi_1 \end{pmatrix} \wedge \begin{pmatrix} \chi^1 \\ \chi^2 \\ \chi^3 \end{pmatrix} + \begin{pmatrix} \lambda^2 \chi^2 \wedge \chi^3 \\ 0 \\ 0 \end{pmatrix}
\mod{\theta^1, \theta^2, \theta^3},
\]
where
\[ \pi_1 = -\d\lambda, \qquad \pi_2 = 2\lambda \varphi^3 + \lambda^2 \chi^3, \qquad \pi_3 = -\big(2\lambda \varphi^2 + \lambda^2 \chi^2\big). \]
Since $\lambda \neq 0$, the torsion cannot be absorbed and this system has no integral elements, and hence no integral manifolds. Thus we conclude that there are no integral manifolds unless $\mu \neq 0$, and henceforth we assume that this is the case.

Now, differentiating equations \eqref{rank1-generators-for-I}, reducing modulo $\big(\theta^1, \theta^2, \theta^3\big)$, and multiplying on the left by ${}^t\hskip-1pt \eurof$ yields the surprisingly simple formula
\begin{gather}\label{rank1-lambda-nonzero-dthetas}
 {}^t\hskip-1pt \eurof \begin{pmatrix} \d\theta^1 \\ \d\theta^2 \\ \d\theta^3 \end{pmatrix} \equiv
-\begin{pmatrix} \pi_1 & \pi_4 & \pi_5 \\ 2\lambda \pi_2 & -\pi_1 & 0 \\ 2\lambda \pi_3 & -\mu \pi_3 & -\pi_1 + \mu \pi_2 \end{pmatrix} \wedge \begin{pmatrix} \chi^1 \\ \chi^2 \\ \chi^3 \end{pmatrix} \mod{\theta^1, \theta^2, \theta^3},
\end{gather}
where
\begin{gather*}
\pi_1 = -\d\lambda + \mu \varphi^3, \qquad
\pi_2 = \varphi^3 + \tfrac{1}{2}\lambda \chi^3, \qquad
\pi_3 = -\big(\varphi^2 + \tfrac{1}{2}\lambda \chi^2\big), \\
\pi_4 = \d\mu + 2\lambda \varphi^3 + \big(3\lambda^2 + \mu^2\big) \chi^3, \qquad
\pi_5 = -\mu \varphi^1 - 2\lambda \varphi^2.
\end{gather*}
The tableau matrix in equation \eqref{rank1-lambda-nonzero-dthetas} has Cartan characters $s_1 = 3$, $s_2 = 2$, $s_3 = 0$, and the space of integral elements is 6-dimensional, parametrized by
\begin{gather*}
\pi_1 = - 2\lambda p_1 \chi^1 + \big(2\mu p_1 + \mu^2 p_2\big) \chi^2, \\
\pi_2 = 2\lambda p_2 \chi^1 + p_1 \chi^2, \\
\pi_3 = 2\lambda p_3 \chi^1 - \mu p_3 \chi^2 + (p_1 + \mu p_2) \chi^3, \\
\pi_4 = \big(2\mu p_1 + \mu^2 p_2\big) \chi^1 + p_4 \chi^2 + p_5 \chi^3, \\
\pi_5 = p_5 \chi^2 + p_6 \chi^3,
\end{gather*}
with $p_1, p_2, p_3, p_4, p_5, p_6 \in \bbR$. Since $s_1 + 2 s_2 + 3 s_3 = 7 > 6$, the system $\cI$ is not involutive, and we need to prolong.

After some rearranging, we can parametrize the space of integral elements of $\cI$ more ma\-na\-geably for computational purposes as
\begin{gather*}
\d\lambda = 2 \lambda u_3 \chi^1 - \mu u_3 \chi^2 - \tfrac{1}{2}\lambda \mu \chi^3, \\
\d\mu = \big(2\mu u_3 - \big(4\lambda^2 + \mu^2\big)u_4\big) \chi^1 + u_1 \chi^2 - \big(\mu u_5 + \lambda^2 + \mu^2\big) \chi^3, \\
\varphi^1 = 2\lambda^2 u_6 \chi^1 + (u_5 - \lambda\mu u_6) \chi^2 + u_2 \chi^3, \\
\varphi^2 = -\lambda \mu u_6 \chi^1 + \tfrac{1}{2}\big(\mu^2 u_6 - \lambda\big) \chi^2 - u_3 \chi^3, \\
\varphi^3 = 2\lambda u_4 \chi^1 + (u_3 - \mu u_4) \chi^2 - \tfrac{1}{2}\lambda \chi^3,
\end{gather*}
with $u_1, u_2, u_3, u_4, u_5, u_6 \in \bbR$. The prolongation $\cI^{(1)}$ of $\cI$ is the exterior differential system on the manifold $X^{(1)} = X \times \bbR^6$, with coordinates $(u_1,\ldots, u_6)$ on the $\bbR^6$ factor, generated by the 1-forms $\big(\theta^1, \theta^2, \theta^3\big)$, together with the 1-forms
\begin{gather}
\theta^4 = \d\lambda - 2 \lambda u_3 \chi^1 + \mu u_3 \chi^2 + \tfrac{1}{2}\lambda \mu \chi^3, \nonumber\\
\theta^5 = \d\mu - \big(2\mu u_3 - \big(4\lambda^2 + \mu^2\big)u_4\big) \chi^1 - u_1 \chi^2 + \big(\mu u_5 + \lambda^2 + \mu^2\big) \chi^3, \nonumber\\
\theta^6 = \varphi^1 - 2\lambda^2 u_6 \chi^1 - (u_5 - \lambda\mu u_6) \chi^2 - u_2 \chi^3, \nonumber\\
\theta^7 = \varphi^2 + \lambda \mu u_6 \chi^1 - \tfrac{1}{2}\big(\mu^2 u_6 - \lambda\big) \chi^2 + u_3 \chi^3, \nonumber\\
\theta^8 = \varphi^3 - 2\lambda u_4 \chi^1 - (u_3 - \mu u_4) \chi^2 + \tfrac{1}{2}\lambda \chi^3.\label{rank1-forms-in-I1}
\end{gather}

From this point on, the details of the computation become rather unwieldy, so we will just give a sketch of the next few steps.\footnote{All computations were carried out with the assistance of MAPLE, including the Cartan package which was written by the second author and is available at \url{http://math.colorado.edu/~jnc/Maple.html}.}
Computing the 2-forms $\big(\d\theta^4,\ldots,\d\theta^8\big)$ and reducing modulo the 1-forms $\big(\theta^1,\ldots,\theta^8\big)$ yields a system for which the torsion cannot be absorbed~-- and hence there are no integral elements~-- except along the codimension~1 submanifold $X' \subset X^{(1)}$ defined by the equation
\begin{gather}\label{rank1-torsion1}
2\lambda u_1 - 2\lambda \mu u_2 - 4\lambda^2 u_3 - \big(4 \lambda^2 \mu + \mu^3\big) u_4 = 0.
\end{gather}
Thus any integral manifold of the system $\cI^{(1)}$ on $X^{(1)}$ must be contained in $X'$.

We may parametrize the solution space to equation \eqref{rank1-torsion1} by
\begin{gather}
u_1 = \mu v_1 + 2\lambda v_2 +\big(4\lambda^2\mu + \mu^3\big) v_4, \qquad u_2 = v_1, \qquad u_3 = v_2, \nonumber\\
u_4 = 2\lambda v_4, \qquad u_5 = v_3, \qquad u_6 = v_5,\label{rank1-torsion-param}
\end{gather}
with $v_1, v_2, v_3, v_4, v_5 \in \bbR$. Substituting the expressions \eqref{rank1-torsion-param} into equations \eqref{rank1-forms-in-I1} yields a new EDS $\cI'$ on $X' \cong X \times \bbR^5$ with the property that the integral manifolds of $\cI^{(1)}$ are precisely the integral manifolds of the system $\cI'$ on $X'$.

Now computing the 2-forms $\big(\d\theta^4,\ldots,\d\theta^8\big)$ and reducing modulo the 1-forms $\big(\theta^1,\ldots,\theta^8\big)$ yields a system for which the torsion can be absorbed. The tableau matrix has Cartan characters $s_1 = 5$, $s_2 = s_3 = 0$, but the space of integral elements is only 4-dimensional. Since $s_1 + 2 s_2 + 3 s_3 = 5 > 4$, the system $\cI'$ is not involutive, and so we need to prolong again. The prolongation~$\cI'^{(1)}$ is the EDS on the manifold $X'^{(1)} = X' \times \bbR^4$, with coordinates $(w_1,\ldots, w_4)$ on the $\bbR^4$ factor, generated by the 1-forms $\big(\theta^1,\ldots, \theta^8\big)$, together with the $1$-forms
\begin{gather}
\theta^9 = \pi_9 + w_1 \chi^2 - w_2 \chi^3, \qquad
\theta^{10} = \pi_{10}, \qquad
\theta^{11} = \pi_{11} + w_2 \chi^2 + w_1 \chi^3, \nonumber\\
\theta^{12} = \pi_{12} - 4\lambda^2 w_4 \chi^1 + 2\lambda \mu w_4 \chi^2, \qquad
\theta^{13} = \pi_{13} + 2\lambda w_3 \chi^1 - \mu w_3 \chi^2,\label{rank1-forms-in-I'1}
\end{gather}
where, for each $j=1,\ldots, 5$, the $1$-form $\pi_{j+8}$ has the form
\[ \pi_{j+8} = \d v_j - P_{jk} \chi^k, \]
and the functions $P_{jk}$ are polynomials in $(v_1,\ldots, v_5)$ with coefficients that are rational functions of $\lambda$ and $\mu$ with nonvanishing denominators.

Computing the 2-forms $\big(\d\theta^9,\ldots,\d\theta^{13}\big)$ and reducing modulo the 1-forms $\big(\theta^1, \ldots, \theta^{13}\big)$ yields a system for which the torsion cannot be absorbed---and hence there are no integral elements~-- except along the codimension 2 submanifold $X'' \subset X'^{(1)}$ defined by two independent equations that are linear in the variables $(w_1, w_2, w_3, w_4)$. These equations can be solved for $w_3$ and $w_4$, yielding expressions of the form
\begin{gather}
w_3 = \frac{8\mu}{3\big(4\lambda^2 + \mu^2\big)^2} (4\lambda v_4 w_1 - \mu v_5 w_2 ) + \frac{1}{\lambda \mu \big(4\lambda^2 + \mu^2\big)^3} P_3, \nonumber\\
w_4 = -\frac{\mu^2}{3\lambda^2 \big(4\lambda^2 + \mu^2\big)^2} (\mu v_5 w_1 + 4 \lambda v_4 w_2 ) + \frac{1}{\lambda^3 \big(4\lambda^2 + \mu^2\big)^3} P_4,\label{w3w4}
\end{gather}
where $P_3$ and $P_4$ are polynomials in the variables $(\lambda, \mu, v_1, \ldots, v_5)$.
Substituting the expres\-sions~\eqref{w3w4} into equations \eqref{rank1-forms-in-I'1} yields a new EDS $\cI''$ on $X'' \cong X' \times \bbR^2$ with the property that the integral manifolds of $\cI'^{(1)}$ are precisely the integral manifolds of the system $\cI''$ on $X''$.

Now computing the 2-forms $\big(\d\theta^4,\ldots,\d\theta^8\big)$ and reducing modulo the 1-forms $\big(\theta^1,\ldots,\theta^8\big)$ yields a system of the form
\begin{gather}
\d\theta^i \equiv 0 \mod{\theta^1, \ldots, \theta^{13}}, \qquad 1 \leq i \leq 8, \nonumber\\
\begin{pmatrix} \d\theta^9 \\ \d\theta^{10} \\ \d\theta^{11} \\ \d\theta^{12} \\ \d\theta^{13} \end{pmatrix} \equiv
\begin{pmatrix}
0 & \pi_{14} & -\pi_{15} \\
0 & 0 & 0 \\
0 & \pi_{15} & \pi_{14} \\
\displaystyle{\frac{4\mu^2}{3 \big(4\lambda^2 + \mu^2\big)^2}} \pi_{16} &
\displaystyle{-\frac{2 \mu^3}{3\lambda \big(4\lambda^2 + \mu^2\big)^2}} \pi_{16}
& 0 \\
\displaystyle{\frac{16\lambda\mu}{3\big(4\lambda^2 + \mu^2\big)^2}} \pi_{17} &
\displaystyle{-\frac{8\mu^2}{3\big(4\lambda^2 + \mu^2\big)^2}} \pi_{17}
& 0
\end{pmatrix} \wedge \begin{pmatrix} \chi^1 \\ \chi^2 \\ \chi^3 \end{pmatrix}\nonumber \\
\hphantom{\begin{pmatrix} \d\theta^9 \\ \d\theta^{10} \\ \d\theta^{11} \\ \d\theta^{12} \\ \d\theta^{13} \end{pmatrix} \equiv}{} +
\begin{pmatrix} T^9_{jk} \chi^j \wedge \chi^k \vspace{1mm}\\ T^{10}_{jk} \chi^j \wedge \chi^k \vspace{1mm}\\ T^{11}_{jk} \chi^j \wedge \chi^k \vspace{1mm}\\ T^{12}_{jk} \chi^j \wedge \chi^k \vspace{1mm}\\ T^{13}_{jk} \chi^j \wedge \chi^k \end{pmatrix} \mod{\theta^1, \ldots, \theta^{13}}, \label{ugly-tableau}
\end{gather}
where
\[
\left. \begin{array}{@{}l@{}}
\pi_{14} \equiv \d w_1 \\
\pi_{15} \equiv \d w_2
\end{array} \right\} \mod{\chi^1, \chi^2, \chi^3}
\]
and
\[ \pi_{16} = \mu v_5 \pi_{14} + 4 \lambda v_4 \pi_{15} , \qquad \pi_{17} = 4\lambda v_4 \pi_{14} - \mu v_5 \pi_{15}. \]

First, consider the open set where $v_4^2 + v_5^2 \neq 0$. On this open set, the 1-forms~$\pi_{16}$ and~$\pi_{17}$ are linearly independent linear combinations of the 1-forms $\pi_{14}$ and $\pi_{15}$, and the torsion terms $T^i_{jk} \chi^j \wedge \chi^k$ cannot be absorbed except along a codimension 1 submanifold defined by a complicated polynomial equation. Moreover, the form of the tableau matrix in equation~\eqref{ugly-tableau} implies that $\cI''$ possesses a unique integral element at each point of this submanifold. This means that the restriction of $\cI''$ to this submanifold is, at best, a Frobenius system with a finite-dimensional space of integral manifolds. More likely, differentiating the equation that defines this submanifold will lead to additional relations that will further restrict the set that admits integral elements, thereby reducing the dimension of the space of integral manifolds, possibly to the point that there are no integral manifolds on which $v_4^2 + v_5^2 \neq 0$. Unfortunately, we have not been able to carry out this computation to completion, so we will content ourselves with the statement that the space of integral manifolds on which $v_4^2 + v_5^2 \neq 0$ is at most finite-dimensional.

Next, we consider the case where $v_4 = v_5 = 0$. In order to characterize integral manifolds satisfying this condition, we must go back to the system $\cI'$ on the manifold $X'$ generated by $\big(\theta^1, \ldots, \theta^8\big)$ and restrict to the codimension 2 submanifold $Y \subset X'$ defined by the equations $v_4 = v_5 = 0$. Let $\cJ$ denote the restriction of $\cI'$ to $Y$; then $\cJ$ is generated by the $1$-forms $\big(\theta^1, \theta^2, \theta^3\big)$, together with the $1$-forms
\begin{gather*}
\begin{aligned}
\theta^4 & = \d\lambda - 2 \lambda v_2 \chi^1 + \mu v_2 \chi^2 + \tfrac{1}{2}\lambda \mu \chi^3, \\
\theta^5 & = \d\mu - 2\mu v_2 \chi^1 - (\mu v_1 + 2\lambda v_2 ) \chi^2 + \big(\mu v_3+ \lambda^2 + \mu^2\big) \chi^3, \\
\theta^6 & = \varphi^1 - v_3 \chi^2 - v_1 \chi^3, \\
\theta^7 & = \varphi^2 + \tfrac{1}{2} \lambda \chi^2 + v_2 \chi^3, \\
\theta^8 & = \varphi^3 - v_2 \chi^2 + \tfrac{1}{2}\lambda \chi^3.
\end{aligned}
\end{gather*}

Computing the 2-forms $\big(\d\theta^4,\ldots,\d\theta^8\big)$ and reducing modulo the 1-forms $\big(\theta^1, \ldots, \theta^8\big)$ yields a~system for which the torsion can be absorbed. The tableau matrix has Cartan characters $s_1 = 3$, $s_2 = s_3 = 0$, but the space of integral elements is only 2-dimensional. Since $s_1 + 2 s_2 + 3 s_3 = 3 > 2$, the system $\cJ$ is not involutive, and we need to prolong. The prolongation $\cJ^{(1)}$ is the EDS on the manifold $Y^{(1)} = Y \times \bbR^2$, with coordinates $(q_1, q_2)$ on the $\bbR^2$ factor, generated by the 1-forms $\big(\theta^1,\ldots, \theta^8\big)$, together with the $1$-forms
\begin{gather*}
\theta^9 = \d v_1 + \big(\tfrac{1}{2} \lambda v_3 - v_1 v_2 \big) \chi^1 + q_1 \chi^2 + q_2 \chi^3 , \\
\theta^{10} = \d v_2 + \big(\tfrac{1}{4} \lambda^2 - v_2^2 \big) \chi^1 - \tfrac{1}{4} \lambda \mu \chi^2 + \tfrac{1}{2} \mu v_2 \chi^3 , \\
\theta^{11} = \d v_3 - \big( \tfrac{1}{2} \lambda v_1 + v_2 v_3 \big) \chi^1 - \left( q_2 - \mu v_1 + 2 \frac{\lambda}{\mu}(\lambda v_1 - 2 v_2 v_3 )\right) \chi^2 \\
\hphantom{\theta^{11} =}{} + \big( q_1 + v_1^2 + v_2^2 + v_3^2 + \mu v_3 + \tfrac{1}{4} \lambda^2 \big) \chi^3.
\end{gather*}

Computing the 2-forms $\big(\d\theta^9,\d\theta^{10},\d\theta^{11}\big)$ and reducing modulo the 1-forms $\big(\theta^1, \ldots, \theta^{11}\big)$ yields a system for which the torsion can be absorbed. The tableau matrix has Cartan characters $s_1 = 2$, $s_2 = s_3 = 0$, and the space of integral elements at each point is 2-dimensional. Since $s_1 + 2 s_2 + 3 s_3 = 2$, the system $\cJ^{(1)}$ is involutive, with integral manifolds locally parametrized by 2 functions of 1 variable.

As a result of this computation and Remark~\ref{lambda-remark}, we have the following theorem.

\begin{Theorem}\label{lambda-nonzero-function-count-theorem} Aside from a possible finite-dimensional family of solutions $($which may be empty$)$, the space of local orthonormal coframings $\big(\omega^1, \omega^2, \omega^3\big)$ on an open subset of $\bbR^3$ whose exterior derivatives $\big(\d\omega^1, \d\omega^2, \d\omega^3\big)$ are pairwise linearly dependent and do not simultaneously vanish and satisfy the additional property that
\[ \omega^1 \wedge \d \omega^1 + \omega^2 \wedge \d \omega^2 + \omega^3 \wedge \d \omega^3 \neq 0 \]
is locally parametrized by $2$ functions of $1$ variable.
\end{Theorem}

One consequence of this result is that the space of integral manifolds with $\lambda \neq 0$ is strictly smaller than the space of integral manifolds with $\lambda = 0$, which we recall is locally parametrized by 1 function of 2 variables. In particular, if the function $\mbz\big(u^1, u^2\big)$ is specified in advance, there will be no solutions with $\lambda \neq 0$ for generic choices of $\mbz$. The question of precisely which choices for the function $\mbz\big(u^1, u^2\big)$ do admit solutions is an interesting one, but we shall not attempt to address it here.

\subsection[Explicit solutions with $\lambda = 0$]{Explicit solutions with $\boldsymbol{\lambda = 0}$}\label{explicit-solns-sec}

We will conclude by showing how to construct explicit solutions with $\lambda=0$ for arbitrary choices of the function $\mbz\big(u^1, u^2\big)$ that satisfy a certain nondegeneracy condition, which will be described below. First, we will show how to construct local coordinates and a local normal form for a general integral manifold of the system $\bar{\cI}$ on the manifold $\bar{X}$. We will need the following well-known fact from linear algebra:

\begin{Lemma}\label{rank2-lemma}
Let $\mbv = {}^t\hskip-1pt\big(v^1, v^2, v^3\big)$ be a nonzero vector in $\bbR^3$, and let $[\mbv]$ denote the skew-symmetric matrix
\[ [\mbv] = \begin{pmatrix} 0 & v^3 & -v^2 \\-v^3 & 0 & v^1 \\ v^2 & -v^1 & 0 \end{pmatrix}. \]
Then $[\mbv]$ has rank $2$, and its kernel is spanned by $\mbv$. Specifically, for any vector $\mbw \in \bbR^3$, we have
\[ {}^t\hskip-1pt \mbw [\mbv] = [\mbv] \mbw = 0 \]
if and only of $\mbw$ is a scalar multiple of $\mbv$.
\end{Lemma}

Let $N^3 \subset \bar{X}$ be any integral manifold of $\bar{\cI}$; in keeping with our conventions, let $\omega$ and $\phi$ denote the pullbacks to $N$ of $\xi$ and~$\alpha$, respectively. As noted above, the assumption that $\lambda=0$ implies that the map $\mba\colon V \to \SO(3)$ whose graph determines the integral manifold $N$ has rank~1. Therefore, there exists a local coordinate function~$u^1$ on $V$ such that $\mba = \mba\big(u^1\big)$, and we can write
\begin{gather}\label{lambda-zero-matrix-f-eqn}
[\phi] = \mba^{-1} \d\mba = \begin{pmatrix}
0 & g^3\big(u^1\big) & -g^2\big(u^1\big) \\ -g^3\big(u^1\big) & 0 & g^1\big(u^1\big) \\ g^2\big(u^1\big) & -g^1\big(u^1\big) & 0 \end{pmatrix} \d u^1
\end{gather}
for some smooth functions $g^i\big(u^1\big)$ on $V$ that do not all vanish simultaneously.

Let $\mbg\big(u^1\big)$ denote the $\bbR^3$-valued function $\mbg\big(u^1\big) = {}^t\hskip-1pt\big(g^1\big(u^1\big), g^2\big(u^1\big), g^3\big(u^1\big)\big)$. From equa\-tion~\eqref{lambda-zero-matrix-f-eqn}, the $\bbR^3$-valued $2$-form $\Omega$ must satisfy
\begin{gather}\label{invariant-Omega}
 \Omega = \d\omega = -[\phi] \wedge \omega = -\big[\mbg\big(u^1\big)\big] \d u^1 \wedge \omega.
\end{gather}
It follows that each of the 2-forms $\big(\Omega^1, \Omega^2, \Omega^3\big)$ must have the 1-form $\d u^1$ as a factor. By Darboux's theorem, we can find another independent coordinate function $u^2$ on $V$ such that each of the 2-forms $\Omega^i$ is a multiple of $\d u^1 \wedge \d u^2$.

Now let $u^3$ be any coordinate function on $V$ that is independent from $u^1$ and $u^2$, so that $\big(u^1, u^2, u^3\big)$ form a local coordinate system on $V$. Let $\mbu = \big(u^1, u^2, u^3\big)\colon V \to \bbR^3$ and let $U = \mbu(V) \subset \bbR^3$; then we may regard $\big(u^1, u^2, u^3\big)$ as local coordinates on~$N$ and $\mbx$ and $\mba$ as functions $\mbx\colon U \to \bbR^3$ and $\mba\colon U \to \SO(3)$.

Next, we can write
\[ \omega = \mbw_j \d u^j \]
for some $\bbR^3$-valued functions $(\mbw_1, \mbw_2, \mbw_3)$ on $U$ that are linearly independent at each point of~$U$. Then we have
\[ -[\phi] \wedge \omega = -\big[\mbg\big(u^1\big)\big] \mbw_2 \d u^1 \wedge \d u^2 -
\big[\mbg\big(u^1\big)\big] \mbw_3 \d u^1 \wedge \d u^3. \]
Since the left-hand side is a multiple of $\d u^1 \wedge \d u^2$, it follows that $\big[\mbg\big(u^1\big)\big] \mbw_3 = 0$. Since the vector~$\mbw_3$ cannot vanish, it must lie in the kernel of the rank 2 matrix $\big[\mbg\big(u^1\big)\big]$; therefore, Lemma~\ref{rank2-lemma} implies that
\[ \mbw_3 = \bar{\mu}\big(u^1, u^2, u^3\big) \mbg\big(u^1\big) \]
for some smooth, nonvanishing function $\bar{\mu}\big(u^1, u^2, u^3\big)$.
Setting
\[ \tilde{u}^3 = \int \bar{\mu}\big(u^1, u^2, u^3\big) \d u^3, \]
we have
\begin{gather*}
\mbw_3 \d u^3 = \mbg\big(u^1\big) \bar{\mu}\big(u^1, u^2, u^3\big) \d u^3
 \equiv \mbg\big(u^1\big) \d \tilde{u}^3 \mod{\d u^1, \d u^2}.
\end{gather*}
So, via the local coordinate transformation $\big(u^1, u^2, u^3\big) \to \big(u^1, u^2, \tilde{u}^3\big)$, we can arrange that $\mbw_3 = \mbg\big(u^1\big)$.

We now have
\begin{gather}\label{omega-partial-soln1}
 \omega = \mbw_1 \d u^1 + \mbw_2 \d u^2 + \mbg\big(u^1\big) \d u^3.
\end{gather}
Differentiating gives
\begin{gather}
\Omega = \d\omega = -(\mbw_2)_3 \d u^2 \wedge \d u^3 +
\big( (\mbw_1)_3 - \mbg'\big(u^1\big) \big) \d u^3 \wedge \d u^1 \nonumber\\
\hphantom{\Omega = \d\omega =}{} + \big( (\mbw_2)_1 - (\mbw_1)_2) \big) \d u^1 \wedge \d u^2,\label{diff-omega-1}
\end{gather}
where subscripts outside parentheses indicate partial derivatives with respect to the coordinates $u^i$. On the other hand, substituting~\eqref{omega-partial-soln1} into \eqref{invariant-Omega} yields
\begin{gather}\label{diff-omega-2}
\Omega = -\big[\mbg\big(u^1\big)\big] \mbw_2 \d u^1 \wedge \d u^2.
\end{gather}
Comparing \eqref{diff-omega-1} and \eqref{diff-omega-2} yields the differential equations
\begin{gather}\label{w-pdes}
(\mbw_2)_3 = 0, \qquad (\mbw_1)_3 = \mbg'\big(u^1\big), \qquad (\mbw_2)_1 - (\mbw_1)_2 = -\big[\mbg\big(u^1\big)\big] \mbw_2.
\end{gather}
The first two equations in \eqref{w-pdes} imply that $\mbw_1$, $\mbw_2$ have the form
\[ \mbw_1 = u^3 \mbg'\big(u^1\big) + \mbh_1\big(u^1, u^2\big), \qquad \mbw_2 = \mbh_2\big(u^1, u^2\big) \]
for some $\bbR^3$-valued functions $\mbh_1$, $\mbh_2$ of $\big(u^1, u^2\big)$ alone. Then the third equation in \eqref{w-pdes} implies that
\[ (\mbh_1)_2 = (\mbh_2)_1 + \big[\mbg\big(u^1\big)\big] \mbh_2. \]
The general solution to this equation is
\[ \mbh_1 = (\mbk)_1 + \big[\mbg\big(u^1\big)\big] \mbk, \qquad \mbh_2 = (\mbk)_2, \]
where $\mbk\big(u^1, u^2\big)$ is an arbitrary, smooth $\bbR^3$-valued function of $\big(u^1, u^2\big)$.

We now have
\begin{gather}
\omega = \big( \big(u^3 \mbg'\big(u^1\big) + \mbk_1\big(u^1, u^2\big) + \big[\mbg\big(u^1\big)\big] \mbk\big(u^1, u^2\big) \big) \d u^1 \nonumber\\
\hphantom{\omega =}{} + \mbk_2\big(u^1, u^2\big) \d u^2 + \mbg\big(u^1\big) \d u^3,\label{lambda-zero-define-omega}
\end{gather}
where $\mbk_i\big(u^1, u^2\big)$ denotes $\frac{\partial}{\partial u^i} \big(\mbk\big(u^1, u^2\big)\big)$.
(Note that $\mbk\big(u^1, u^2\big)$ must be chosen so that the components of $\omega$ are linearly independent at each point of~$U$.)
Moreover, we have
\begin{gather*}
\Omega = -[\phi] \wedge \omega = -\big[\mbg\big(u^1\big)\big] \mbk_2\big(u^1, u^2\big) \d u^1 \wedge \d u^2.
\end{gather*}

Now, suppose that we are given a vector $\Omega$ of closed 2-forms on $U$ whose components $\big(\Omega^1, \Omega^2, \Omega^3\big)$ are all scalar multiples of a single 2-form and do not vanish simultaneously. What conditions must $\Omega$ satisfy in order to guarantee the existence of a local coordinate system $\big(u^i\big)$ on $U$ and $\bbR^3$-valued functions $\mbg\big(u^1\big)$, $\mbk\big(u^1, u^2\big)$ so that the coframing $\omega$ given by~\eqref{lambda-zero-define-omega} satisfies the condition $\d \omega = \Omega$?

First note that, by Darboux's theorem, we can find local coordinates $\big(u^1, u^2, u^3\big)$ on $U$ such that
\[ \Omega = \mbz\big(u^1, u^2\big) \d u^1 \wedge \d u^2 \]
for some smooth, nonvanishing $\bbR^3$-valued function $\mbz\big(u^1, u^2\big)$. Moreover, under any change of coordinates of the form $\big(u^1, u^2, u^3\big) \to \big(\tilde{u}^1(u^1, u^2\big), \tilde{u}^2\big(u^1, u^2\big), u^3\big)$, each of the functions $z^i\big(u^1, u^2\big)$ is multiplied by the determinant of the Jacobian of the coordinate transformation. Thus, it is geometrically natural to regard $\mbz$ as defining a map $[\![\mbz]\!]$ into~$\bbR \bbP^2$, and this map is unchanged by coordinate transformations of this form.

The following theorem shows that a mild nondegeneracy condition on the function $[\![\mbz]\!]$ is sufficient to guarantee the existence of solutions.

\begin{Theorem}\label{big-theorem}Let $\Omega = \mbz\big(u^1, u^2\big) \d u^1 \wedge \d u^2$, where $\mbz\colon M \to \bbR^3\setminus \{\mathbf{0}\}$ is a smooth, nonvanishing function. Let $[\![\mbz]\!]\colon M \to \bbR \bbP^2$ denote the composition of $\mbz\colon M \to \bbR^3 \setminus\{\mathbf{0}\}$ with the standard projection $\bbR^3\setminus\{\mathbf{0}\} \to \bbR \bbP^2$, and suppose that either, $(i)$ the image of $[\![\mbz]\!]$ is contained in a line in $\bbR \bbP^2$, or $(ii)$ $d[\![\mbz]\!]$ is nonvanishing on $M$. Then every point of $M$ has a neighborhood $U$ on which there exist functions $\mba\colon U \to \SO(3)$, $\mbx\colon U \to \bbR^3$ such that the map $\mba$ has rank~$1$ and the components $\big(\omega^1, \omega^2, \omega^3\big)$ of the $\bbR^3$-valued $1$-form $\omega = \mba^{-1} \d\mbx$ form a local coframing on $U$ and $\d\omega = \Omega$.
\end{Theorem}

\begin{proof}We will show that, possibly after a coordinate transformation of the form $\big(u^1, u^2, u^3\big) \to \big(\tilde{u}^1\big(u^1, u^2\big), \tilde{u}^2\big(u^1, u^2\big), u^3\big)$, we can find $\bbR^3$-valued functions $\mbg\big(u^1\big)$, $\mbk\big(u^1, u^2\big)$ such that
\begin{gather}\label{Z-soln}
 -\big[\mbg\big(u^1\big)\big] \mbk_2\big(u^1, u^2\big) = \mbz\big(u^1, u^2\big).
\end{gather}
It is important to observe that the matrix $\big[\mbg\big(u^1\big)\big]$ necessarily
has rank~2, and equation~\eqref{Z-soln} requires that, for any fixed value of~$u^1$, the vector $\mbz\big(u^1, u^2\big)$ be contained in the image of $\big[\mbg\big(u^1\big)\big]$ for all values of $u^2$. This, in turn, is true if and only if
\begin{gather}\label{vanishing-condition}
{}^t\hskip-1pt \mbg\big(u^1\big) \mbz\big(u^1, u^2\big) = 0.
\end{gather}
There may not initially appear to exist such a function $\mbg\big(u^1\big)$ depending on $u^1$ alone, but under the hypotheses of the theorem, we can find refined local coordinates and a nonvanishing func\-tion~$\mbg\big(u^1\big)$ for which this condition holds. For instance:
\begin{itemize}\itemsep=0pt
\item If the image of $[\![\mbz]\!]$ is contained in a line in $\bbR \bbP^2$, then there exist constants $a_1, a_2, a_3 \in \bbR$, not all zero, such that
\[ a_1 z^1\big(u^1, u^2\big) + a_2 z^2\big(u^1, u^2\big) + a_3 z^3\big(u^1, u^2\big) = 0. \]
In this case, let $\mbg = {}^t\hskip-1pt (a_1 , a_2, a_3 )$.
\item If $\d[\![\mbz]\!]$ is nonvanishing on $U$, then in some neighorhood of every point, at least one of the ratios $z^i/z^j$ has no critical points. If, say, the ratio $r\big(u^1, u^2\big) = z^2/z^1$ is nonconstant and has no critical points, then in a neighborhood of any point we can make a change of coordinates of the form $\big(u^1, u^2, u^3\big) \to \big(\tilde{u}^1\big(u^1, u^2\big), \tilde{u}^2\big(u^1, u^2\big), u^3\big)$ with $\tilde{u}^1 = r\big(u^1, u^2\big)$, so that in the new coordinates we have $z^2 = u^1 z^1$. After performing this coordinate transformation, let $\mbg\big(u^1\big) = {}^t\hskip-1pt \big( u^1, -1, 0 \big)$.
\end{itemize}
Now, having constructed the desired local coordinate system and function $\mbg\big(u^1\big)$, let $\mbk_2\big(u^1, u^2\big)$ be a smooth solution of the linear system of equations~\eqref{Z-soln}. As noted above, this equation can be solved for $\mbk_2\big(u^1, u^2\big)$ precisely because the condition~\eqref{vanishing-condition} is exactly the condition required to ensure that for every $\big(u^1, u^2\big)$, the vector $\mbz\big(u^1, u^2\big)$ lies in the image of the rank 2 matrix~$\big[\mbg\big(u^1\big)\big]$.

Now let
\begin{gather}\label{construct-k}
 \mbk\big(u^1, u^2\big) = \int \mbk_2\big(u^1, u^2\big) \d u^2 + \bar{\mbk}\big(u^1\big),
\end{gather}
where the function $\bar{\mbk}\big(u^1\big)$ may be chosen arbitrarily,
and define $\omega$ by equation~\eqref{lambda-zero-define-omega}. By construction, $\omega$ satisfies $\d\omega = \Omega$ and so is the desired coframing.

The only detail remaining to check is that the components $\mbw_j$ of $\omega$ in~\eqref{lambda-zero-define-omega} are linearly independent, so that $\big(\omega^1, \omega^2, \omega^3\big)$ is a coframing on~$U$. First, observe that $\mbw_3 = \mbg\big(u^1\big)$ lies in the kernel of $\big[\mbg\big(u^1\big)\big]$. The vector $\mbw_2 = \mbk_2\big(u^1, u^2\big)$, however, must satisfy \eqref{Z-soln} and so cannot lie in the kernel of $\big[\mbg\big(u^1\big)\big]$; hence the vectors $\mbw_2$ and $\mbw_3$ are linearly independent. And since the function $\bar{\mbk}\big(u^1\big)$ in \eqref{construct-k} may be chosen arbitrarily, we can arrange for~$\mbw_1$ to be linearly independent from $\mbw_2$ and $\mbw_3$ by choosing $\bar{\mbk}\big(u^1\big)$ appropriately.

Finally, the functions $\mba\colon U \to \SO(3)$ and $\mbx\colon U \to \bbR^3$ promised by the theorem may be constructed as follows. First, the function $\mba\colon U \to \SO(3)$ is given by the solution (unique up to multiplication by a constant matrix in $\SO(3)$) of the ODE
\begin{gather}\label{a-ODE}
\mba'\big(u^1\big) = \mba\big(u^1\big) \big[\mbg\big(u^1\big)\big].
\end{gather}
Then the function $\mbx\colon U \to \bbR^3$ is given by integrating the (necessarily closed) 1-form
\[ \d\mbx = \mba \omega. \]
Note that, while constructing these functions requires solving the ODE \eqref{a-ODE}, the coframing $\omega$ can be constructed from $\Omega$ using only quadratures.
\end{proof}

The following example shows that the nondegeneracy assumptions of Theorem \ref{big-theorem} are essential; specifically, it shows how the construction above can fail near a point where $\d[\![\mbz]\!]$ vanishes.

\begin{Example}\label{critical-point-example}
For ease of notation, we will use $(u,v)$ in place of $\big(u^1, u^2\big)$ in this example.
Suppose that
\[ \mbz(u,v) = {}^t\hskip-1pt\big(1, \rho(u,v), \rho(u,v)^2\big) , \]
where
\[ \rho(u,v) = u^2 + v^2. \]
Then $\d[\![\mbz]\!](0,0) = 0$.

Suppose that there exists a $(0,0)$-centered local coordinate system $ (\tilde{u}, \tilde{v} )$ in some neighborhood $U$ of $(0,0)$ and a nonvanishing vector field $\mbg(\tilde{u})$ on $U$ such that
\begin{gather}\label{example-3-null-condition}
{}^t\hskip-1pt \mbg(\tilde{u}) \mbz(\tilde{u}, \tilde{v}) = 0.
\end{gather}
Because the function $\rho$ has a critical point at $(0,0)$ and is strictly convex, it has the property that for any nonvanishing vector field $\mbv$ on $U$,
\[ \mbv[\rho](0,0) = 0, \qquad \mbv[\mbv[\rho]](0,0) > 0. \]
In particular, we have
\[ \rho_{\tilde{v}}(0,0) = 0, \qquad \rho_{\tilde{v} \tilde{v}} (0,0) = \kappa_0 > 0, \]
where subscripts denote partial derivatives with respect to $\tilde{v}$ in the $(\tilde{u}, \tilde{v})$ coordinate system.
It follows that
\begin{gather*}
\mbz(0,0) = {}^t\hskip-1pt (1, 0, 0), \\
 \mbz_{\tilde{v} \tilde{v}}(0,0) =
{}^t\hskip-1pt \big(0, \rho_{\tilde{v} \tilde{v}}, 2\big(\rho \rho_{\tilde{v} \tilde{v}} + \rho_{\tilde{v}}^2\big) \big) \big\vert_{(0,0)} = {}^t\hskip-1pt (0, \kappa_0, 0 ), \\
\mbz_{\tilde{v} \tilde{v} \tilde{v} \tilde{v}}(0,0) =
{}^t\hskip-1pt \big(0, \rho_{\tilde{v} \tilde{v} \tilde{v} \tilde{v}},
6 \rho_{\tilde{v} \tilde{v}}^2 + 8 \rho_{\tilde{v}} \rho_{\tilde{v} \tilde{v}\tilde{v}} + 2 \rho \rho_{\tilde{v} \tilde{v} \tilde{v} \tilde{v}} \big) \big\vert_{(0,0)} =
{}^t\hskip-1pt \big(0, *, 6\kappa_0^2\big),
\end{gather*}
where the second entry of $\mbz_{\tilde{v} \tilde{v} \tilde{v} \tilde{v}}(0,0)$ is irrelevant.

Consequently, evaluating equation \eqref{example-3-null-condition} together with its $2$nd and $4$th $\tilde{v}$-de\-ri\-va\-tives at $(\tilde{u},\tilde{v}) = (0,0)$ yields three independent linear equations for the components of $\mbg(0)$. It follows that $\mbg(0) = \mathbf{0}$, and hence there is no nonvanishing vector field $\mbg(\tilde{u})$ satisfying the condition~\eqref{example-3-null-condition} for any local coordinate system $(\tilde{u}, \tilde{v})$ on any neighborhood of $(u,v) = (0,0)$.
\end{Example}

\subsection*{Acknowledgements}
Thanks to Duke University for its support via a research grant (Bryant), to the National Science Foundation for its support via research grant DMS-1206272 (Clelland), and to the Simons Foundation for its support via a Collaboration Grant for Mathematicians (Clelland).

\pdfbookmark[1]{References}{ref}
\LastPageEnding

\end{document}